\documentclass[11pt,fleqn]{article}

\usepackage[left=1.4in,right=1.4in]{geometry}

\usepackage[utf8]{inputenc}
\usepackage[T1]{fontenc}

\usepackage{amsmath,amssymb,amsthm}
\usepackage{mathtools}

\usepackage{tikz-cd}
\usepackage{enumerate}
\usepackage{paralist}

\usepackage{booktabs}
\usepackage{multirow}

\usepackage[hyperfootnotes=false]{hyperref}
\usepackage[capitalize,noabbrev,nameinlink]{cleveref}
\crefformat{equation}{(#2#1#3)}

\usepackage[cal=pxtx,scr=esstix,scrscaled=1.05]{mathalpha} 
\usepackage{aurical} 
\usepackage{upgreek}
\usepackage{fouriernc}

\usepackage{physics}

\usepackage[backend=bibtex,style=alphabetic]{biblatex}
\DeclareFieldFormat{eprint:HAL}{\textsc{hal}: \href{https://tel.archives-ouvertes.fr/#1}{\texttt{#1}}}

\theoremstyle{plain}
\newtheorem{theorem}{Theorem}[subsection]

\newtheorem{lemma}[theorem]{Lemma}
\newtheorem{corollary}[theorem]{Corollary}

\theoremstyle{definition}
\newtheorem{definition}[theorem]{Definition}
\newtheorem{example}[theorem]{Example}

\theoremstyle{remark}
\newtheorem{remark}[theorem]{Remark}

\numberwithin{equation}{theorem}

\newcommand{\cover}{\mathcal{U}}

\newcommand{\OO}{\mathcal{O}}
\newcommand{\anotherbullet}{\star}

\newcommand{\restricted}{\mathbin{\big\vert}}
\newcommand{\id}{\mathrm{id}}
\newcommand{\congto}{\xrightarrow{\raisebox{-.5ex}[0ex][0ex]{$\sim$}}}
\newcommand{\cech}{\check{\mathscr{C}}}

\newcommand{\GL}{\mathrm{GL}}
\newcommand{\at}{\mathrm{at}}
\newcommand{\expat}[1]{\at^{\circ#1}}
\newcommand{\stanat}[1]{\at^{\wedge#1}}
\newcommand{\simpexpat}[1]{\hat{\at}^{\circ#1}}
\newcommand{\simpstanat}[1]{\hat{\at}^{\wedge#1}}

\newcommand{\sym}[1]{S_{#1}}
\newcommand{\sgn}[1]{|#1|}

\newcommand{\gcohX}{\mathsf{Coh}(X)}

\newcommand{\greenzeroX}{\mathsf{Green}_{\nabla,0}(X_\bullet^\cover)}

\newcommand{\sgreenzeroX}{\underline{\mathsf{Green}}_{\nabla,0}(X_\bullet^\cover)}

\newcommand{\define}[1]{\textbf{#1}}
\newcommand{\nerve}[1]{X_{#1}^\cover}

\newcommand{\xtwoheadrightarrow}[2][]{%
  \xrightarrow[#1]{#2}\mathrel{\mkern-14mu}\rightarrow
}

\renewcommand{\d}{\mathrm{d}}
\renewcommand{\matrix}{\mathrm{Mat}}
\renewcommand{\ss}[1]{\varsigma_{#1}}

\DeclareMathOperator{\Hom}{Hom}
\DeclareMathOperator{\Tot}{Tot}
\DeclareMathOperator{\HH}{H}
\DeclareMathOperator{\Ker}{Ker}

\DeclareMathOperator{\LL}{L}
\DeclareMathOperator{\sheafhom}{\mathscr{H}\textup{\kern -2.5pt {\large\Fontauri\slshape om}}\,}
\DeclareMathOperator{\sheafend}{\mathscr{E}\textup{\kern -2pt {\large\Fontauri\slshape nd}}\,}
\DeclareMathOperator{\footnotesizesheafend}{\mathscr{E}\textup{\kern -2pt {\Fontauri\slshape nd}}\,}

\DeclareMathOperator{\Ext}{Ext}

\bibliography{part-II.bib}

\title{Simplicial Chern-Weil theory for\\coherent analytic sheaves\\{\Large Part II: {\em Barycentric connections}}}
\author{Timothy Hosgood}
\date{}

\begin{document}

\maketitle

\begin{abstract}
    In the previous part of this diptych, we defined the notion of an \emph{admissible simplicial connection}, as well as explaining how H.I. Green constructed a resolution of coherent analytic sheaves by locally free sheaves on the Čech nerve.
    This paper seeks to apply these abstract formalisms, by showing that Green's \emph{barycentric} simplicial connection is indeed admissible, and that this condition is exactly what we need in order to be able to apply Chern-Weil theory and construct characteristic classes.
    We show that, in the case of (global) vector bundles, the simplicial construction agrees with what one might construct manually: the explicit Čech representatives of the exponential Atiyah classes of a vector bundle agree.
    Finally, we summarise how all the preceding theory fits together to allow us to define Chern classes of coherent analytic sheaves, as well as showing uniqueness in the compact case.
\end{abstract}

\tableofcontents

\bigskip

\begin{center}
    \emph{This project has received funding from the European Research Council (ERC) under the European Union’s Horizon 2020 research and innovation programme.}

    \emph{(Grant Agreement No. 768679)}
\end{center}

\medskip

{This paper is one of two (the other being \cite{Hosgood2021}) to have been extracted from the author's PhD thesis \cite{Hosgood2020}. Further details and more lengthy exposition can be found there. We repeat here, however, our genuine thanks to Julien Grivaux and Damien Calaque for their tireless tutelage.}
We also sincerely thank the reviewer, whose comments and numerous corrections greatly improved the content and presentation of this work.

\section{Introduction}

    \subsection{History and motivation}

        For an in-depth history and motivation, we refer the reader to the previous part of this diptych \cite[§1]{Hosgood2021} (and we more generally assume that the reader is already somewhat familiar with its contents, although we always make explicit mention of any results that we use).
        We will however give a brief overview here.

        In 1980, Green's thesis \cite{Green1980} gave a construction of Chern classes for coherent analytic sheaves using twisting cochains, technology developed and put to great use by Toledo, Tong, and O'Brian in a series of work throughout the 70s and 80s, where they proved, amongst other things, a version of Grothendieck--Riemann--Roch for complex manifolds \cite{OTT1985}.
        Although Green's work was summarised from a more abstract point of view in \cite{TT1986}, there was still no formal framework for the ``simplicial connections'' of which he made use, nor for his resolution, which turned a twisting cochain (an inherently homotopical object) into a complex of vector bundles on the Čech nerve (something much more homotopically strict).
        The main purpose of \cite{Hosgood2021} was to rectify this situation, using the language of $(\infty,1)$-categories to turn Green's resolution into an equivalence between certain categories of complexes of coherent sheaves and that of vector bundles on the Čech nerve endowed with simplicial connections, and also to formally define the ingredients necessary to develop a simplicial version of Chern--Weil theory.
        The purpose of this present work is to show that explicit calculations one can do ``by hand'', Green's construction, and this simplicial Chern--Weil theory, all give the same Čech representatives for the Chern classes of coherent analytic sheaves.

    \subsection{Purpose and overview}

        The calculations in this paper largely focus on \emph{global}\footnote{That is, vector bundles on the nerve that are given by the pullback of vector bundles on $X$.} vector bundles on the nerve; we delay the study of coherent sheaves and arbitrary vector bundles on the nerve until the last section.
        This is because Green vector bundles on the nerve (resp. simplicial connections generated in degree zero) are really a mild generalisation of pullbacks of global vector bundles (resp. Green's {barycentric connections}), and so, by studying the latter, we can mostly understand the former.

        The main contributions of this paper are the following:
        \begin{enumerate}
            \item We manually compute explicit representatives in the Čech-de~Rham bicomplex of the first four Atiyah classes (from which we can recover the Chern classes) of an arbitrary holomorphic vector bundle on a complex manifold (\cref{section:manual-construction}).
            \item We introduce and develop a \emph{simplicial} version of Chern--Weil theory, and show that this can be applied in particular to Green's barycentric connection, placing it in a general formal framework.
                Explicit calculations show that this agrees with the manual computations above (\cref{section:simplicial-construction-for-vector-bundles}).
            \item This simplicial Chern--Weil theory allows us to construct Čech simplicial Atiyah classes of complexes of coherent analytic sheaves.
                We then show that, in the case of a \emph{compact} complex manifold, we recover exactly Čech representatives of the classical Chern classes (\cref{corollary:main-corollary}).
        \end{enumerate}

        In \cref{section:preliminaries} we recall the definition of \define{Atiyah classes} and study some of their properties, before recalling the definition of \define{fibre integration} of simplicial differential forms.

        The purpose of \cref{section:manual-construction} is simply to detail an inefficient algorithm for lifting exponential Atiyah classes of vector bundles to closed elements in the Čech-de~Rham bicomplex.
        This isn't particularly useful on its own, but will later act as an assurance that the simplicial construction does indeed given the results that we would hope for.
        Indeed, this section demonstrates the need for a more abstract approach.

        As promised in the previous part of this diptych, we explain, in \cref{section:characteristic-classes-via-admissible-simplicial-connections}, how the admissibility condition of simplicial connections is exactly what we need in order to be able to apply \define{Chern-Weil theory}, and thus get well-defined characteristic classes from our simplicial connections.
        
        Although the preceding section is general enough to apply to arbitrary vector bundles on the nerve, in \cref{section:simplicial-construction-for-vector-bundles}, we apply the results to the specific case of Green's \define{barycentric connection} on global vector bundles.
        This allows us to compare the resulting Čech representatives to those calculated in \cref{section:manual-construction} to see that they do indeed agree.

        Finally, \cref{section:from-vector-bundles-to-coherent-sheaves} gives a summary of the whole story contained in this two-part work: we already know that we can resolve coherent analytic sheaves by vector bundles on the nerve, and that we even get sufficiently nice connections on the latter; now we can apply the results in this paper to obtain characteristic classes.
        There is an even stronger result, further reinforcing the fact that this simplicial construction is indeed the ``good'' one: whenever $X$ is compact, we can appeal to an axiomatisation of Chern classes to show that the construction given here agrees with any other construction that one might construct in some other way.

\section{Preliminaries}\label{section:preliminaries}

    Throughout, let $(X,\OO_X)$ be a paracompact complex-analytic manifold with its structure sheaf the sheaf of holomorphic functions; let $\cover$ be a locally-finite Stein open cover such that all finite intersections are also Stein.
    Let $E$ be a vector bundle of rank $\mathfrak{r}$ on $X$, and assume that it is trivialised by $\cover$.
    We have trivialisation maps $\varphi_\alpha\colon E\restricted U_\alpha\congto(\OO_X\restricted U_\alpha)^\mathfrak{r}$, and transition maps
    \[
        M_{\alpha\beta}\colon(\OO_X\restricted{U_{\alpha\beta}})^\mathfrak{r}\congto(\OO_X\restricted{U_{\alpha\beta}})^\mathfrak{r}
    \]
    given on overlaps by \mbox{$M_{\alpha\beta}=\varphi_\alpha\circ\varphi_\beta^{-1}$}.
    By picking some basis of sections $\{s^{\alpha}_1,\ldots,s^{\alpha}_\mathfrak{r}\}$ of $E$ over $U_\alpha$ we can realise the $M_{\alpha\beta}$ as $(\mathfrak{r}\times\mathfrak{r})$-matrices that describe the change of basis when we go from $E\restricted U_\beta$ to $E\restricted U_\alpha$, i.e.
    \begin{equation}
    \label{equation:transition-maps}
        s^\alpha_k = \sum_\ell(M_{\alpha\beta})_k^\ell s^\beta_\ell.
    \end{equation}

    We continue to use the conventions, definitions, and notation from the previous paper \cite{Hosgood2021}, but recall particularly pertinent results when necessary.
    In addition, we write $\sym{k}$ to mean the symmetric group on $k$ elements, and $\sgn{\sigma}$ to mean the sign of a permutation $\sigma\in\sym{k}$.

    \subsection{Standard and exponential Atiyah classes}

        \begin{definition}
            The \define{Atiyah exact sequence} (or \define{jet sequence}) of a locally free sheaf $E$ of $\OO_X$-modules is the short exact sequence
            \[
                0
                \to
                E\otimes_{\OO_X}\Omega_X^1
                \to
                J^1(E)
                \to
                E
                \to
                0
            \]
            where $J^1(E)=(E\otimes\Omega_X^1)\oplus E$ as a $\mathbb{C}_X$-module (writing $\mathbb{C}_X$ to mean the constant sheaf on $X$ of value $\mathbb{C}$), but with an $\OO_X$-action defined by
            \[
                f(s\otimes\omega,t)
                =
                (fs\otimes\omega + t\otimes\d f, ft).
            \]

            The \define{Atiyah class} $\at_E$ of $E$ is defined to be the extension class of this short exact sequence:
            \[
                \at_E
                =
                [J^1(E)] \in \Ext_{\OO_X}^1(E,E\otimes\Omega_X^1).
            \]
        \end{definition}

        \begin{remark}
            We are interested in the Atiyah class, as well as the higher \emph{standard} and \emph{exponential} Atiyah classes, (which we define later), of a vector bundle because their trace is ``equivalent'', in some sense that we make precise below, to the Chern classes.
            \begin{itemize}
                \item \cite[Proposition~4.3.10]{Huybrechts2005} shows us that the the trace of the Atiyah class gives the same class in cohomology as the curvature of the Chern connection;
                combined with \cite[Example~4.4.8~i)]{Huybrechts2005}, this tells us that the standard Atiyah classes and the Chern classes are the same, up to a constant.
                    There are some more general comments about this equivalence in \cite[p.~200]{Huybrechts2005}, just below Example~4.4.11.
                \item \cite[Exercise~4.4.11]{Huybrechts2005} tells us that the Chern characters (or exponential Chern classes) are given (up to a constant) by the traces of the exponential Atiyah classes.
                \item The traces of the Atiyah classes satisfy an axiomatisation of Chern classes that, \emph{in the case where $X$ is compact}, guarantees uniqueness.
                    This is explained in \cref{subsection:the-compact-case}.
            \end{itemize}

            Note that the treatment of the Atiyah class in \cite{Huybrechts2005} is somehow reverse to that found here: it starts there with the Atiyah class as a Čech cocycle and \emph{then} shows that it is equivalent to some splitting of the Atiyah exact sequence.
        \end{remark}

        \begin{definition}
            A \define{holomorphic (Koszul) connection} $\nabla$ on a vector bundle $E$ on $X$ is a (holomorphic) splitting of the Atiyah exact sequence of $E$.
            By enforcing the \define{Leibniz rule}
            \[
                \nabla(s\otimes\omega)
                =
                \nabla s\wedge\omega + s\otimes\d\omega
            \]
            we can extend any connection $\nabla\colon E\to E\otimes\Omega_X^1$ to a map\footnote{Using the same symbol $\nabla$ to denote the connection as well as any such extension is a common abuse of notation.} $\nabla\colon E\otimes\Omega_X^r\to E\otimes\Omega_X^{r+1}$.
            The \define{curvature} $\kappa(\nabla)$ of a connection $\nabla$ is the map
            \[
                \kappa(\nabla)=\nabla\circ\nabla\colon
                E\to E\otimes\Omega_X^2
            \]
            given by enforcing the above Leibniz rule.

            Given a connection $\nabla$ on $E$, we say that a section $s\in\Gamma(U,E)$ is \define{flat} if $\nabla(s)=0$; we say that the connection $\nabla$ is \define{flat} if $\kappa(\nabla)=0$.
        \end{definition}

        \begin{remark}
            Given the definition of a connection as a splitting of the Atiyah exact sequence, and the Atiyah class as the extension class of the Atiyah exact sequence, we see that one way of understanding the Atiyah class is as \emph{the obstruction to admitting a (global) holomorphic connection}.
        \end{remark}

        \begin{lemma}
            Any vector bundle on a Stein manifold admits a holomorphic connection.
            \begin{proof}
                This is an application of Cartan's Theorem B, using the fact that a vector bundle can be viewed as a locally free sheaf.
                See \cite[Lemma~O.E.3]{Green1980}.
            \end{proof}
        \end{lemma}

        \begin{lemma}\label{lemma:cech-cocycle-of-atiyah-class}
            The Atiyah class of $E$ is represented by the Čech cocycle
            \[
                \left\{
                    \nabla_\beta \restricted U_{\alpha\beta}
                    -
                    \nabla_\alpha \restricted U_{\alpha\beta}
                \right\}_{\alpha,\beta}
                \in
                \cech_\cover^1\left(
                    \sheafhom\left(E,E\otimes\Omega_X^1\right)
                \right).
            \]
            \begin{proof}
                First, recall that the difference of any two connections is exactly an $\OO_X$-linear map.
                Secondly, note that the cocycle condition is indeed satisfied, since
                \begin{equation*}
                    (\nabla_\beta-\nabla_\alpha) + (\nabla_\gamma-\nabla_\beta) = \nabla_\gamma-\nabla_\alpha.
                \end{equation*}
                Combined, this tells us that $\{\nabla_\beta-\nabla_\alpha\}_{\alpha,\beta}\in\cech_\cover^1(\sheafhom(E,E\otimes\Omega_X^1))$.
                Then we use the isomorphisms
                \begin{equation*}
                    \Ext_{\OO_X}^1(E,E\otimes\Omega_X^1) \cong \Hom_{\mathcal{D}(X)}(E,E\otimes\Omega_X^1[1]) \cong \HH^1\big(X,\sheafhom(E,E\otimes\Omega_X^1)\big).
                \end{equation*}
                Finally, we have to prove that the class thus defined in homology agrees with that given in our definition of the Atiyah class.
                We show this by proving a more general fact.

                Let $0\to \mathcal{A}\to \mathcal{B}\to \mathcal{C}\to 0$ be a short exact sequence in some abelian category $\mathscr{A}$.
                By definition, $[\mathcal{B}]\in\Ext_\mathscr{A}^1(\mathcal{C},\mathcal{A})$ is the class in $\Hom_{\mathcal{D}(\mathscr{A})}(\mathcal{C},\mathcal{A}[1])$ of the canonical morphism $\mathcal{C}\to\mathcal{A}[1]$ constructed using $\mathcal{B}$ as follows:
                \begin{enumerate}[(i)]
                    \item consider $\mathcal{A}\to\mathcal{B}$ as a complex with $\mathcal{B}$ in degree~$0$, and take the quasi-isomorphism $(\mathcal{A}\to \mathcal{B})\congto \mathcal{C}$;
                    \item invert this quasi-isomorphism to get a map $\mathcal{C}\congto(\mathcal{A}\to \mathcal{B})$ such that the composite $\mathcal{C}\congto(\mathcal{A}\to \mathcal{B})\to\mathcal{B}\to\mathcal{C}$ is the identity;
                    \item compose with the identity map $(\mathcal{A}\to \mathcal{B})\to \mathcal{A}[1]$.
                \end{enumerate}
                When $\mathscr{A}$ is the category of locally free sheaves on $X$, we can realise the quasi-isomorphism \mbox{$\mathcal{C}\congto(\mathcal{A}\to \mathcal{B})$} as a quasi-isomorphism $\mathcal{C}\congto\cech^\bullet(\mathcal{A}\to\mathcal{B})$ by using the Čech complex of a complex:
                \begin{equation*}
                    \cech^\bullet(\mathcal{A}\to \mathcal{B}) = \cech^0(\mathcal{A}) \xrightarrow{(\check\delta,f)} \cech^1(\mathcal{A})\oplus\cech^0(\mathcal{B}) \xrightarrow{(\check\delta,-f,\check\delta)} \cech^2(\mathcal{A})\oplus\cech^1(\mathcal{B}) \xrightarrow{(\check\delta,f,\check\delta)} \ldots
                \end{equation*}
                where $\cech^0(\mathcal{A})$ is in degree $-1$.
                If we have local sections $\sigma_\alpha\colon \mathcal{C}\restricted{U_\alpha}\to \mathcal{B}\restricted{U_\alpha}$ then $\sigma_\beta-\sigma_\alpha$ lies in the kernel $\Ker(\mathcal{B}\restricted{U_{\alpha\beta}}\to \mathcal{C}\restricted{U_{\alpha\beta}})$, and so we can lift this difference to $\mathcal{A}$, giving us the map
                \begin{equation*}
                    (\{\sigma_\alpha\}_\alpha, \{\sigma_\beta-\sigma_\alpha\}_{\alpha,\beta})\colon \mathcal{C} \to \cech^0(\mathcal{B})\oplus\cech^1(\mathcal{A}).
                \end{equation*}
                This map we have constructed is exactly $[\mathcal{B}]$.
                More precisely,
                \begin{equation*}
                    \begin{array}{rcccl}
                        \Ext_{\OO_X}^1(\mathcal{C},\mathcal{A}) &\cong& \Hom_{\mathcal{D}(X)}(\mathcal{C},\mathcal{A}[1]) &\cong& \HH^1\big(X,\sheafhom(\mathcal{C},\mathcal{A})\big)\\[.8em]
                        [\mathcal{B}] &\leftrightarrow& \mathcal{C}\congto(\mathcal{A}\to \mathcal{B})\to \mathcal{A}[1] &\leftrightarrow& [\{\sigma_\beta-\sigma_\alpha\}_{\alpha,\beta}].
                    \end{array}\qedhere
                \end{equation*}
            \end{proof}
        \end{lemma}

        \begin{definition}\label{definition:omega-alpha-beta}
            We write $\omega_{\alpha\beta}$ for the cocycle $(\nabla_\beta \restricted U_{\alpha\beta} - \nabla_\alpha \restricted U_{\alpha\beta})\in\cech_\cover^1(\sheafhom(E,E\otimes\Omega_X^1))$.
        \end{definition}

        \begin{remark}
            When $\mathcal{E}$, $\mathcal{F}$, and $\mathcal{G}$ are sheaves of $\OO_X$-modules, with $\mathcal{G}$ locally free, we have the isomorphism
            \begin{equation*}
                \sheafhom(\mathcal{E},\mathcal{F}\otimes\mathcal{G}) \cong \sheafhom(\mathcal{E},\mathcal{F})\otimes\mathcal{G}.
            \end{equation*}
            In particular,
            \begin{equation*}
                \HH^1\big(X,\sheafhom(E,E\otimes\Omega_X^1)\big) \cong \HH^1\big(X,\sheafend(E)\otimes\Omega_X^1\big).
            \end{equation*}
            Taking the trivialisation over $U_\alpha$, this lets us consider $\omega_{\alpha\beta}$ as an \define{endomorphism-valued $1$-form}: an $(\mathfrak{r}\times\mathfrak{r})$-matrix of (holomorphic) $1$-forms on $X$ (where $\mathfrak{r}$ is the rank of $E$).
        \end{remark}

        \begin{remark}\label{remark:cup-product-cech}
            Recall that, for sheaves $\mathscr{F}$ and $\mathscr{G}$ of $\OO_X$-modules, we have the cup product
            \begin{equation*}
                \smile \colon \HH^m(X,\mathscr{F})\otimes\HH^n(X,\mathscr{G}) \to \HH^{m+n}(X,\mathscr{F}\otimes\mathscr{G})
            \end{equation*}
            which is given in Čech cohomology by the tensor product: $(a\smile b)_{\alpha\beta\gamma} = (a)_{\alpha\beta}\otimes(b)_{\beta\gamma}$.
        \end{remark}

        \begin{definition}\label{definition:second-exponential-atiyah-class}
            We define the \define{second \emph{exponential}\footnote{In general, we will be interested in \emph{standard} Atiyah classes, but we consider the manual construction of \emph{exponential} ones since these can be written down very explicitly, as shown in \cref{section:manual-construction}.} Atiyah class} $\expat{2}_E$ by the following construction:
            take the cup product
            \begin{equation*}
                (\at_E\otimes\,\id_{\Omega_X^1})\smile\at_E \in \HH^2\big(X,\sheafhom(E\otimes\Omega_X^1,E\otimes\Omega_X^1\otimes\Omega_X^1)\otimes\sheafhom(E,E\otimes\Omega_X^1)\big)
            \end{equation*}
            and apply the composition map
            \begin{equation*}
                \HH^m\big(X,\sheafhom(\mathcal{F},\mathcal{G})\otimes\sheafhom(\mathcal{E},\mathcal{F})\big) \to \HH^m\big(X,\sheafhom(\mathcal{E},\mathcal{G})\big)
            \end{equation*}
            to obtain
            \begin{align*}
                (\at_E\otimes\,\id_{\Omega_X^1})\smile\at_E &\in \HH^2\big(X,\sheafhom\left(E,E\otimes\Omega_X^1\otimes\Omega_X^1\right)\big)\\
                &\cong \HH^2\big(X,\sheafend(E)\otimes\Omega_X^1\otimes\Omega_X^1\big)
            \end{align*}
            and then apply the wedge product (of differential forms) to get
            \begin{equation*}
                \expat{2}_E=(\at_E\otimes\,\id_{\Omega_X^1})\wedge(\at_E) \in \HH^2\big(X,\sheafend(E)\otimes\Omega_X^2)\big).\qedhere
            \end{equation*}
            
            In general, the \define{$k$-th exponential Atiyah class} $\expat{k}_E$ is the class
            \begin{equation*}
                \expat{k}_E = \prod_{i=1}^k\left(\at_E\otimes\id_{\Omega_X^1}^{\otimes(k-i)}\right) \in \HH^k\big(X,\sheafend(E)\otimes\Omega_X^k)\big)
            \end{equation*}
            where the product is given by applying composition and then the wedge product (of forms) as above.
        \end{definition}
            
        \begin{remark}
            As a general note on notation, we will omit the wedge symbol $\wedge$ when talking about the wedge product of differential forms (or we will use $\cdot$ if we have to use any symbol at all), and reserve it solely for the wedge product of endomorphisms.
            In particular, for endomorphism-valued forms $M$ and $N$, we write $MN$ (or \mbox{$M\cdot N$}) to mean the object given by composing the endomorphisms and wedging the forms, and $M\wedge N$ to mean the object given by wedging the endomorphisms (following \cite[Definition~2.1.1]{Ble1981}) and wedging the forms.
            In terms of $(2\times2)$-matrices (i.e. taking $E$ to be of rank 2), this means
            \begin{align*}
                \begin{pmatrix}
                    a&b\\
                    c&d
                \end{pmatrix}\cdot
                \begin{pmatrix}
                    e&f\\
                    g&h
                \end{pmatrix}
                &= 
                \begin{pmatrix}
                    ae+bg&af+bh\\
                    ce+dg&cf+dh
                \end{pmatrix}
                \in\Gamma\big(U,\Omega_X^2\otimes\sheafend(E)\big)\\
                \begin{pmatrix}
                    a&b\\
                    c&d
                \end{pmatrix}\wedge
                \begin{pmatrix}
                    e&f\\
                    g&h
                \end{pmatrix}
                &= 
                \det\begin{pmatrix}
                    a&f\\
                    c&h
                \end{pmatrix}=ah-fc+bg-de
                \in\Gamma\big(U,\Omega_X^2\otimes\sheafend(E\wedge E)\big)
                \cong\Gamma\big(U,\Omega_X^2\big).
            \end{align*}
            Note that, if we take the trace, then these two objects will both be $2$-forms on $U$.
            Similarly, $\tr(M^k)$ and $\tr(\wedge^kM)$ are both just $k$-forms on $U$.
        \end{remark}

        \begin{remark}
            The classical theory of Chern classes has two important ``types'' of Chern class: standard and exponential.
            For now, we are content with simply saying, as a definition, that the polynomial that gives the exponential classes is $\tr(M^p)$, and the polynomial that gives the standard classes is $\tr(\wedge^pM)$.
            Caution is needed when discussing the $k$-th Atiyah class though: there is \emph{no} trace in the definition; to obtain characteristic classes we \emph{have} to take the trace.
        \end{remark}

        \begin{definition}\label{definition:second-standard-atiyah-class}
            We define the \define{second \emph{standard} Atiyah class} $\stanat{2}_E$ by the following construction:
            take the cup product
            \begin{align*}
                \at_E\smile\at_E &\in \HH^2\big(X,\sheafhom(E,E\otimes\Omega_X^1)\otimes\sheafhom(E,E\otimes\Omega_X^1)\big)\\
                &\cong \HH^2\big(X,\sheafend(E)\otimes\sheafend(E)\otimes\Omega_X^1\otimes\Omega_X^1\big)
            \end{align*}
            and then apply the wedge product of endomorphisms and the wedge product of forms to get
            \begin{equation*}
                \stanat{2}_E \in \HH^2\big(X,\sheafend(E\wedge E)\otimes\Omega_X^2\big).
            \end{equation*}
            
            In general, the \define{$k$-th standard Atiyah class} $\stanat{k}_E$ is the class
            \begin{equation*}
                \stanat{k}_E=\bigwedge_{i=1}^k\at_E \in \HH^k\big(X,\sheafend(\wedge^kE)\otimes\Omega_X^k\big).
            \end{equation*}
        \end{definition}

        \begin{remark}\label{remark:have-to-change-base-for-higher-atiyah-class}
            We can find an explicit representative for $\expat{2}_E$ by using \cref{remark:cup-product-cech}:
            \[
                \big((\at_E\otimes\id_{\Omega_X^1})\smile(\at_E)\big)_{\alpha\beta\gamma}
                =
                \left(\at_E\otimes\id_{\Omega_X^1}\right)_{\alpha\beta}\otimes(\at_E)_{\beta\gamma}
            \]
            which corresponds to
            \[
                \left(\omega_{\alpha\beta}\otimes\id_{\Omega_{U_{\alpha\beta}}^1}\right)\otimes\omega_{\beta\gamma}
                \in
                \matrix_{\mathfrak{r}\times\mathfrak{r}}\left(
                    \Omega^1(U_{\alpha\beta}) \otimes \Omega^1(U_{\alpha\beta})
                \right) \otimes \matrix_{\mathfrak{r}\times\mathfrak{r}}\left(
                    \Omega^1(U_{\beta\gamma})
                \right)
            \]
            where $\matrix_{\mathfrak{r}\times\mathfrak{r}}(A)$ is the collection of $A$-valued $(\mathfrak{r}\times \mathfrak{r})$-matrices.
            But before composing these two matrices, as described in \cref{definition:second-exponential-atiyah-class}, we first have to account for the change of trivialisation from $U_{\beta\gamma}$ to $U_{\alpha\beta}$.
            That is, after applying composition and the wedge product, we have
            \begin{equation*}
                \left(\expat{2}_{E}\right)_{\alpha\beta\gamma} = \omega_{\alpha\beta}\wedge M_{\alpha\beta}\omega_{\beta\gamma}M_{\alpha\beta}^{-1}.
            \end{equation*}
        \end{remark}

        \begin{remark}\label{remark:order-of-base-change-and-composition}
            We know that $\expat{3}_E$ is represented locally by $\omega_{\alpha\beta}\omega_{\beta\gamma}\omega_{\gamma\delta}$, but where $\omega_{\beta\gamma}$ and $\omega_{\gamma\delta}$ undergo a base change to become $\Omega^1(U_{\alpha\beta})$-valued.
            But then there seem to be two choices of how we might calculate this:
            \begin{enumerate}[1.]
                \item base change $\omega_{\gamma\delta}$ to be $\Omega^1(U_{\beta\gamma})$-valued; compose with $\omega_{\beta\gamma}$; and \emph{then} base change this composition to be $\Omega^1(U_{\alpha\beta})$-valued; or
                \item base change both $\omega_{\gamma\delta}$ and $\omega_{\beta\gamma}$ to be $\Omega^1(U_{\alpha\beta})$-valued; and \emph{then} perform the triple composition.
            \end{enumerate}
            These two would give the same result if
            \[
                \omega_{\alpha\beta} \wedge M_{\alpha\beta}(\omega_{\beta\gamma} \wedge M_{\beta\gamma}\omega_{\gamma\delta}M_{\beta\gamma}^{-1})M_{\alpha\beta}^{-1}
                =
                \omega_{\alpha\beta} \wedge M_{\alpha\beta}\omega_{\beta\gamma}M_{\alpha\beta}^{-1} \wedge M_{\alpha\gamma}\omega_{\gamma\delta}M_{\alpha\gamma}^{-1}
            \]
            and, thankfully, this equality does indeed hold, thanks to the cocycle condition on the $M_{\alpha\beta}$ and some form of associativity\footnote{That is, $A\cdot MB = AM\cdot B$, where $M$ is a matrix of $0$-forms.}, and so we can use whichever one we so please.
        \end{remark}

    \subsection{Fibre integration}
    \label{subsection:fibre-integration}

        Here we recall some results summarised in \cite[§2.3]{Hosgood2021}.
        Let $Y_\bullet$ be a simplicial complex manifold.
        Following \cite{Dupont1976}, we define a \define{simplicial differential $r$-form on $Y_{\bullet}$} to be a family $(\omega_p)_{p\in\mathbb{N}}$ of forms, with $\omega_p$ a global section of the sheaf
        \[
            \bigoplus_{i+j=r}
                \pi_{Y_p}^* \Omega_{Y_p}^i
                \otimes_{\OO_{Y_p\times \Delta^p_\mathrm{extd}}}
                \pi_{\Delta^p_\mathrm{extd}}^* \Omega_{\Delta^p_\mathrm{extd}}^j
        \]
        (where $\Delta^p_\mathrm{extd}$ is the affine subspace of $\mathbb{R}^{p+1}$ given by the vanishing of $1-\sum_{m=0}^p x_p$; where $\Omega_{Y_p}$ is the sheaf of \emph{holomorphic} forms, and $\Omega_{\Delta^p_\mathrm{extd}}$ is the sheaf of \emph{smooth} forms; and where $\pi_{Y_p}$ and $\pi_{\Delta^p_\mathrm{extd}}$ are the projection maps from $Y_p\times\Delta^p_\mathrm{extd}$) such that, for all \emph{coface} maps $f_p^i\colon[p-1]\to[p]$,
        \begin{equation}
        \label{equation:simplicial-gluing-condition-for-forms}
            \left(Y_\bullet f_p^i\times\id\right)^*\omega_{p-1}
            = \left(\id\times f_p^i\right)^*\omega_p
            \in \Omega^r(Y_{p}\times\Delta^{p-1}).
        \end{equation}

        We write $\Omega^{r,\Delta}(Y_\bullet)$ to mean the algebra of all simplicial differential $r$-forms on $Y_\bullet$.
        We can describe each $\omega_p$ as a form \define{of type $(i,j)$}, by writing $\omega_p = \xi_p\otimes\tau_p$, where $\xi_p$ is the $Y_p$-part of $\omega_p$, and $\tau_p$ is the $\Delta^p$-part of $\omega_p$, and then setting $i=|\xi_p|$ and $j=|\tau_p|$.
        This lets us define a \define{differential}
        \[
            \d \colon \Omega^{r,\Delta}(Y_\bullet) \longrightarrow \Omega^{r+1,\Delta}(Y_\bullet)
        \]
        which is given by the Koszul convention with respect to the type of the form:
        \begin{align*}
            \d(\xi_p\otimes\tau_p)
            &= \left(\d_{Y_\bullet} + (-1)^{|\xi_p|}\d_{\Delta^\bullet}\right) (\xi_p\otimes\tau_p)
        \\  &= \d\xi_p\otimes\tau_p + (-1)^{|\xi_p|}\xi_p\otimes\d\tau_p.
        \end{align*}

        \begin{lemma}[Dupont's fibre integration]\label{lemma:dupont's-fibre-integration}
            There is a {quasi-isomorphism} which, for each fixed degree $r$, consists of a map
            \begin{equation}
                \int_{\Delta^\bullet}\colon \Omega^{r,\Delta}(Y_\bullet)
                \to
                \bigoplus_{p=0}^r\Omega^{r-p}(Y_p)
            \end{equation}
            induced by \define{fibre integration}
            \begin{equation}
                \int_{\Delta^p}\colon\Omega^{r,\Delta}(Y_\bullet)
                \to
                \Omega^{r-p}(Y_p)
            \end{equation}
            where the latter is given by integrating the type $(r-p,p)$ part of a simplicial form over the geometric realisation of the $p$-simplex with its canonical orientation.

            In particular, taking $Y_\bullet=\nerve{\bullet}$ gives
            \[
                \int_{\Delta^\bullet}\colon \Omega^{r,\Delta}(\nerve{\bullet})
                \to
                \bigoplus_{p=0}^r\Omega^{r-p}(\nerve{p})
                \cong
                \Tot^r\cech^\bullet(\Omega^\bullet_X).
            \]
        \end{lemma}

        \begin{remark}\label{remark:fibre-integration-only-depends-on-certain-parts}
            (\cite[p.~36]{Green1980}).
            Since the integral of a $k$-form over an $\ell$-dimensional manifold is only non-zero when $k=\ell$, we see that the fibre integral of some simplicial differential $r$-form $\omega_\bullet = \{\omega_p^{i,j}\}_{p\in\mathbb{N},i+j=r}$ is determined entirely by the type-$(r-p,p)$ parts on the $p$-simplices:
            \begin{equation*}
                \int_{\Delta^\bullet}\omega = \int_{\Delta^0}\omega_0^{r,0} + (-1)^{(r-1)}\int_{\Delta^1}\omega_1^{r-1,1} + \ldots + (-1)^{(r-p)(p)}\int_{\Delta^p}\omega_p^{r-p,p} + \ldots + \int_{\Delta^r}\omega_r^{0,r}
            \end{equation*}
            where the signs come from the fact that we work with $\nerve{\bullet}\times\Delta^\bullet$ instead of $\Delta^\bullet\times\nerve{\bullet}$.
        \end{remark}

        \begin{remark}
            There are further sign confusions that can arise from the choice of orientation on the simplices, but this will not concern us here.
            This is discussed in \cite{Hosgood2020}.
        \end{remark}

\section{Manual construction for global vector bundles}\label{section:manual-construction}

    Assume now that we have \emph{flat} holomorphic (local) connections $\nabla_\alpha$ on each $E\restricted U_\alpha$.

    \subsection{Truncated de~Rham cohomology}

        \begin{remark}\label{remark:we-can-calculate-in-cech}
            Since $\cover$ is Stein (and any finite intersection of these opens is also Stein) and the sheaves $\Omega_X^r$ are coherent, we can apply Cartan's Theorem B, which tells us that $\check{\mathbb{H}}^k(\cover,\Omega_X^\bullet) \cong \check{\mathbb{H}}^k(X,\Omega_X^\bullet)$.
            Since $X$ is paracompact and $\cover$ is Stein, we know that Čech cohomology computes hypercohomology, and so $\check{\mathbb{H}}^k(X,\Omega_X^\bullet) \cong \mathbb{H}^k(X,\Omega_X^\bullet)$.
            Finally, $\mathbb{H}^k(X,\Omega_X^\bullet) \cong \HH^k\left(X,\mathbb{C}\right)$ (by e.g. \cite[Theorem~8.1]{Voisin2002a}).
            In summary, we are in a nice enough setting that Čech-de~Rham bicomplex lets us calculate singular cohomology:
            \begin{equation*}
                \HH^r\Tot^\bullet\cech^\anotherbullet\left(\Omega^\anotherbullet_X\right) \cong \HH^r\left(X,\mathbb{C}\right).
            \end{equation*}
        \end{remark}

        \begin{definition}
            Given the de~Rham complex $\Omega_X^\bullet$, we define the \define{$k$-th Hodge complex} $\Omega_X^{\bullet\geqslant k}$ as the truncation
            \begin{equation*}
                \Omega_X^{\bullet\geqslant k} = \left(\Omega_X^k\xrightarrow{\d}\Omega_X^{k+1}\xrightarrow{\d}\ldots\right)[-k]
            \end{equation*}
            i.e. where $\Omega_X^k$ is in degree $k$.
        \end{definition}

        \begin{definition}
            We define the \define{$k$-th truncated de~Rham cohomology} (or \define{tDR} cohomology) to be $\HH_\text{tDR}^k(X)=\mathbb{H}^k(X,\Omega_X^{\bullet\geqslant k})$.
        \end{definition}

        \begin{remark}\label{remark:when-can-we-refine-to-hodge}
            If we have some \emph{closed} class $c=(c_0,\ldots,c_{2k})\in\Tot^{2k}\cech^\bullet(\Omega^\bullet_X)$, where $c_i\in\cech^i(\Omega_X^{2k-i})$, \emph{such that} $c_i=0$ for $i\geqslant k+1$, then we can refine the corresponding class in singular cohomology $[c]\in\HH^{2k}(X,\mathbb{C})$ to a class $[c]\in\HH_\text{tDR}^{2k}(X)$ in tDR cohomology.
        \end{remark}

        \begin{remark}\label{remark:how-to-manually-lift}
            Recalling \cref{remark:we-can-calculate-in-cech}, we know that we can use the Čech complex to calculate singular cohomology.
            Now, say we are given some $c_k\in\cech^k(\Omega_X^k)$ with $\check{\delta}c_k=0$ but $\d c_k\neq0$.
            \emph{If} we can find $c_i\in\cech^{i}(\Omega^{2k-i}_X)$ for $i=1,\ldots,(k-1)$ such that $\check\delta c_{i-1}=\d c_i$, and define $c_0=0\in\cech^0(\Omega_X^{2k})$, then\footnote{The signs depend on the parity of $k$.}
            \begin{equation*}
                (0,\pm c_1,\ldots,\pm c_{k-1},c_k,0,\ldots,0) \in \Tot^{2k}\cech^\bullet\left(\Omega^\bullet_X\right)
            \end{equation*}
            is $(\check\delta+(-1)^k\d)$-closed, and thus represents a cohomology class in $\HH^{2k}\Tot^\bullet\cech^\anotherbullet(\Omega_X^\anotherbullet)$, and thus a cohomology class in $\HH^{2k}(X,\mathbb{C})$.

            In essence, given some ``starting element'' in the Čech-de~Rham bicomplex, we can try to manually lift it to some \emph{closed} element of the same total degree.
        \end{remark}

        \begin{remark}
            A few important notes before we continue.
            \begin{itemize}
                \item Although we have the isomorphism $H^r\Tot^\bullet\cech^\anotherbullet(\Omega^\anotherbullet_X) \cong \HH^r(X,\mathbb{C})$, we don't necessarily have an easy way of computing explicitly what a closed class in the Čech-de~Rham complex maps to under this isomorphism, \emph{unless} it has a non-zero degree-$(0,r)$ component, in which case it maps to exactly that component.
                \item We will construct classes in singular cohomology, but note that they can actually all be considered as living in the corresponding tDR cohomology, thanks to \cref{remark:when-can-we-refine-to-hodge}.
                \item The assumption that the local connections $\nabla_\alpha$ are \emph{flat} is not necessary for the \emph{abstract} theory, but essential for these \emph{explicit} calculations.
                \item We point out, once more, that the constructions given in this chapter are for the \emph{exponential} Atiyah classes.
            \end{itemize}
        \end{remark}

    \subsection{The first Atiyah class}\label{subsection:first-manual-atiyah-class}

        We wish to calculate $\omega_{\alpha\beta}$ (\cref{definition:omega-alpha-beta}), and we can use the fact that our local sections $\{s_1^\alpha,\ldots,s_\mathfrak{r}^\alpha\}$ are $\nabla_\alpha$-flat:
        \[
            \omega_{\alpha\beta} \left(s_k^\alpha\right)
            =
            \left(\nabla_\beta-\nabla_\alpha\right) \left(s_k^\alpha\right)
            =
            \nabla_\beta \left(s_k^\alpha\right)
        \]
        (where we omit the restriction notation on the local connections).
        But then, using \cref{equation:transition-maps} followed by the Leibniz rule,
        \begin{align*}
            \nabla_\beta\left(s^\alpha_k\right) &= \nabla_\beta\left(\sum_\ell(M_{\alpha\beta})_k^\ell s^\beta_\ell\right)\\
            &=\sum_\ell\left[\nabla_\beta\left(s^\beta_\ell\right)\wedge(M_{\alpha\beta})^\ell_k+s^\beta_\ell\otimes\d(M_{\alpha\beta})^\ell_k\right].
        \end{align*}
        Since the $s^\beta_\ell$ are $\nabla_\beta$-flat, the first part of each $\ell$-term is zero, and then, using the inverse of \cref{equation:transition-maps},
        \begin{align*}
            \omega_{\alpha\beta}\left(s^\alpha_k\right) &= \sum_\ell\left[\left(\sum_m(M_{\alpha\beta}^{-1})^m_\ell s^\alpha_m\right)\otimes\d(M_{\alpha\beta})^\ell_k\right]\\
            &=\sum_m s^\alpha_m\otimes(M_{\alpha\beta}^{-1}\d M_{\alpha\beta})^m_k
        \end{align*}
        where we can move the $M_{\alpha\beta}^{-1}$ across the tensor product because the tensor is over $\OO_X$, and the $M_\ell^m$ are exactly elements of this ring.
        Thus, in the $U_\alpha$ trivialisation,
        \begin{equation}
        \label{equation:omega-alpha-beta-explicitly}
            \omega_{\alpha\beta}
            =
            M_{\alpha\beta}^{-1}\d M_{\alpha\beta}.
        \end{equation}

        \begin{remark}
            This $\omega_{\alpha\beta}$ is exactly the first Chern class $\d\log M_{\alpha\beta}$ of the bundle.
        \end{remark}

        \begin{lemma}\label{lemma:d-omega-is-minus-omega-squared}
            $\d\omega_{\alpha\beta}=-\omega_{\alpha\beta}^2$.
            \begin{proof}
                Using the fact that $\d \left(A^{-1}\right)=-A^{-1}\cdot\d A\cdot A^{-1}$, we see that
                \begin{align*}
                    \d\omega_{\alpha\beta} &= \d\left(M_{\alpha\beta}^{-1}\d M_{\alpha\beta}\right)
                \\  &= \d \left(M_{\alpha\beta}^{-1}\right)\d M_{\alpha\beta}
                \\  &= -M_{\alpha\beta}^{-1}\d M_{\alpha\beta}M_{\alpha\beta}^{-1}\d M_{\alpha\beta}
                \\  &= -\left(M_{\alpha\beta}^{-1}\d M_{\alpha\beta}\right)^2
                \\  &= -\omega_{\alpha\beta}^2.\qedhere
                \end{align*}
            \end{proof}
        \end{lemma}

        \begin{lemma}\label{lemma:d-tr-omega-is-zero}
            $\d\tr\omega_{\alpha\beta}=0$.
            \begin{proof}
                Since $\d\tr\omega_{\alpha\beta}=\tr\d\omega_{\alpha\beta}=-\tr\omega_{\alpha\beta}^2$, by \cref{lemma:d-omega-is-minus-omega-squared}, this lemma is a specific case of the fact\footnote{Using the skew-symmetry of forms, and writing matrix multiplication as a sum, we can show that we can cyclically permute matrices of forms inside a trace, up to a sign ($\tr A_1 A_2\cdots A_k=(-1)^{k-1}\tr A_2\cdots A_k A_1$); this fact is then a corollary, since $\tr(A^{2k})=-\tr(A^{2k})$.} that $\tr(A^{2k})=0$ for any $k\in\mathbb{N}$.
            \end{proof}
        \end{lemma}

        By \cref{lemma:d-tr-omega-is-zero}, and the fact that $\omega_{\alpha\beta}$ is a Čech cocycle by definition (\cref{lemma:cech-cocycle-of-atiyah-class}), we can draw the closed element $\at_E$ in the Čech-de~Rham bicomplex as follows:
        \begin{equation}
            \begin{tikzcd}[row sep=large]
                0&\\
                \underbrace{\tr(\omega_{\alpha\beta})}_{\tr\at_E}
                    \ar[u,"\mathrm{d}"]
                    \ar[r,"\check\delta"]
                & 0
            \end{tikzcd}
        \end{equation}

    \subsection{The second Atiyah class}\label{subsection:second-manual-atiyah-class}

        By \cref{definition:second-exponential-atiyah-class}, we know that
        \begin{equation*}
            \expat{2}_E = \left\{\omega_{\alpha\beta}M_{\alpha\beta}\omega_{\beta\gamma}M_{\alpha\beta}^{-1}\right\}_{\alpha,\beta}
        \end{equation*}
        and we introduce the notation
        \begin{align*}
                A &= \omega_{\alpha\beta}
                \qquad
                M = M_{\alpha\beta}
            \\  B &= \omega_{\alpha\gamma}
                \qquad
                X = M\omega_{\beta\gamma}M^{-1}
        \end{align*}
        so that $\expat{2}_E=AX$, and everything can be thought of as living over $U_\alpha$.

        By \cref{lemma:d-omega-is-minus-omega-squared}, we know that $\mathrm{d}A = -A^2$, and similarly for $B$ and $X$.
        Further, by differentiating the cocycle condition $M_{\alpha\beta}M_{\beta\gamma} = M_{\alpha\gamma}$ of the transition maps, and then right-multiplying by $M_{\alpha\gamma}^{-1}$, we see that\footnote{As previously mentioned, $M\omega_{\beta\gamma}M^{-1}$ is the natural way of thinking of $X$ as being a map \textit{into} something lying over $U_\alpha$, so this equation should be read as a cocycle condition over $U_\alpha$ by thinking of it as $\omega_{\alpha\beta}+\widetilde{\omega}_{\beta\gamma}=\omega_{\alpha\gamma}$, where the tilde corresponds to a base change. Note that this is also the result we expect, since $\omega_{\alpha\beta}$ corresponds to $\nabla_\beta-\nabla_\alpha$, and (as we have already noted) this clearly satisfies the additive cocycle condition.} \mbox{$\tr(A+X)=\tr(B)$}.
        Hence
        \begin{equation*}
            \tr\expat{2}_E = \tr\left(A(B-A)\right).
        \end{equation*}
        Using the fact that $\d A=A^2$ we see that $\d A^2=0$, whence
        \[
            \d\tr\left(\expat{2}_E\right) = \d\tr\left(AB-A^2\right) = -\tr\left(A^2B-AB^2\right).
        \]

        Recalling \cref{remark:how-to-manually-lift}, we want $f\in\cech^1(\Omega_X^3)$ such that $\delta f=-\d\tr(\expat{2}_E)$ and $\d f=0$.
        It is necessary, but not sufficient, that $f$ be the trace of a polynomial of homogeneous degree $3$ in the one variable $A=\omega_{\alpha\beta}$.
        But then $f(A)=\tr(A^3)$ is, up to a scalar multiple, our only option.
        We set $f(A)=\frac13\tr(A^3)$ and compute its Čech coboundary:
        \begin{align*}
            (\delta f)_{\alpha\beta\gamma} &= f(\omega_{\beta\gamma})-f(\omega_{\alpha\gamma})+f(\omega_{\alpha\beta})\\
            &= f(B-A)-f(B)+f(A)\\
            &= \frac13\tr\big((B-A)^3-B^3+A^3\big)\\
            &= \frac13\tr\left(-AB^2 -BAB - B^2A + A^2B + ABA + BA^2\right).
        \end{align*}
        Using the aforementioned (\cref{lemma:d-tr-omega-is-zero}) fact that we can cyclically permute, up to a sign, inside the trace, we see that
        \begin{align*}
            \frac13\tr\left(-AB^2-AB^2-AB^2+A^2B+A^2B+A^2B\right)
            &= \tr\left(A^2B - AB^2\right)\\
            &= -\d\tr\left(\expat{2}_E\right).
        \end{align*}
        Now we just have to worry about whether or not $\d f$ is zero.
        But
        \begin{align*}
            \d f(A) &= \frac13\d\tr(A^3)
        \\  &= \frac13\tr(\d A A^2-A\d A^2)
            = -\frac13\tr(A^4),
        \end{align*}
        and we know, as in \cref{lemma:d-tr-omega-is-zero}, that this is indeed zero.

        This calculation can be summarised by the following diagram:
        \begin{equation}
        \label{equation:second-atiyah-class-summary}
            \begin{tikzcd}[row sep=large]
                0&&\\
                -\frac13\tr\left(A^3\right)
                    \ar[u,"\mathrm{d}"]
                    \ar[r,"\check\delta"]
                &\tr(A(B-A)B)
                &\\
                &\underbrace{\tr\big(A(B-A)\big)}_{\tr\expat{2}_E}
                    \ar[u,"\mathrm{d}"]
                    \ar[r,"\check\delta"]
                & 0
            \end{tikzcd}
        \end{equation}
        which gives us the closed\footnote{See \cref{remark:double-checking-at2-is-closed}.} element
        \begin{equation}
            \left(
                0\,,\,\,
                \frac13\tr(A^3)\,,\,\,
                \tr\big(AX\big)\,,\,\,
                0\,,\,\,
                0
            \right)\in\Tot^4\cech^\bullet\big(\Omega_X^\anotherbullet\big)
        \end{equation}

        \begin{remark}\label{remark:double-checking-at2-is-closed}
            We know that $\tr\expat{2}_E$ is a cocycle by definition, but we can still manually show that it is Čech closed:
            \begin{align*}
                \check\delta\tr\big(\omega_{\alpha\beta}(\omega_{\alpha\gamma}-\omega_{\alpha\beta})\big)
                &= \check\delta\tr(\omega_{\alpha\beta}\omega_{\alpha\gamma})
            \\  &= \tr(\omega_{\beta\gamma}\omega_{\beta\delta}) - \tr(\omega_{\alpha\gamma}\omega_{\alpha\delta}) + \tr(\omega_{\alpha\beta}\omega_{\alpha\delta}) - \tr(\omega_{\alpha\beta}\omega_{\alpha\gamma})
            \\  &= \tr\big((\omega_{\alpha\gamma}-\omega_{\alpha\beta})(\omega_{\alpha\delta}-\omega_{\alpha\beta})\big) - \tr(\omega_{\alpha\gamma}\omega_{\alpha\delta})
            \\  &\qquad + \tr(\omega_{\alpha\beta}\omega_{\alpha\delta}) - \tr(\omega_{\alpha\beta}\omega_{\alpha\gamma})
            \\  &= -\tr(\omega_{\alpha\gamma}\omega_{\alpha\beta}) + \tr(\omega_{\alpha\beta}^2) - \tr(\omega_{\alpha\beta}\omega_{\alpha\gamma})
            \\  &= \tr(\omega_{\alpha\beta}\omega_{\alpha\gamma}) - \tr(\omega_{\alpha\beta}\omega_{\alpha\gamma})
            \\  &= 0.
            \end{align*}
        \end{remark}

    \subsection{The third Atiyah class}\label{subsection:third-manual-atiyah-class}

        We extend our previous notation, writing
        \begin{align*}
            &A=\omega_{\alpha\beta} \qquad M=M_{\alpha\beta} \qquad X=M\omega_{\beta\gamma}M^{-1}\\
            &B=\omega_{\alpha\gamma} \qquad N=M_{\alpha\gamma} \qquad Y=N\omega_{\gamma\delta}N^{-1}\\
            &C=\omega_{\alpha\delta}
        \end{align*}
        so that $\expat{3}_E = AXY = A(B-A)(C-B)$.

        It is then relatively simple to calculate that
        \[
            \d\tr\expat{3}_E = -\tr\big(A(B-A)(C-B)C\big) \in \cech^3\left(\Omega^4_X\right).
        \]
        Trying to find some $\varphi\in\cech^2(\Omega^4_X)$ such that $\check\delta\varphi = \d\tr\expat{3}_E$, however, is slightly harder.
        The most naive approach is to list all the monomials in $\cech^2(\Omega^4_X)$, apply the Čech differential to each one, and then equate coefficients.
        Using the fact that (up to a sign) we can cyclically permute under the trace, finding all the monomials is the same as finding all degree-2 monomials in non-commutative variables $X$ and $Y$, modulo equivalence under cyclic permutation, and there are just four of these: $X^2Y^2$, $(XY)^2$, $X^3Y$, and $XY^3$.
        By calculation,
        \[
            \varphi
            =
            \frac{-1}{4}\tr\Big(\big(A(B-A)\big)^2\Big)
            + \frac12 \tr(A^2(B-A)^2)
            + \frac{-1}{2}\tr\big(A^3(B-A)\big)
            + \frac{-1}{2}\tr\big(A(B-A)^3\big)
        \]
        is exactly such that $\check\delta\varphi = \d\tr\expat{3}_E$.
        Factoring $\d\varphi$, we see that
        \[
            \d\varphi = \frac{1}{10}\tr\big((B-A)^5-B^5+A^5\big) = \check\delta\frac{1}{10}\tr(A^5).
        \]

        This calculation can be summarised in the following diagram:
        \begin{equation}
            \begin{tikzcd}[row sep=huge,column sep=small]
                0\\
                \frac{1}{10}\tr(A^5)
                    \ar{u}{\mathrm{d}} 
                    \ar{r}{\check\delta}
                & 
                \frac{1}{10}\tr\big((B-A)^5 - B^5 + A^5\big)\\
                & -\frac14\tr\Big(\big(A(B-A)\big)^2\Big) + \frac12 \tr\big(A^2(B-A)^2\big)
                    \ar{u}{\mathrm{d}}
                &\\[-4em]
                & - \frac12\tr\big(A^3(B-A)\big) - \frac12\tr\big(A(B-A)^3\big)
                    \ar{r}{\check\delta}
                & -\tr\big(A(B-A)(C-B)C\big) \\
                && \underbrace{\tr\big(A(B-A)(C-B)\big)}_{\tr\expat{3}_E}
                    \ar{u}{\mathrm{d}}
                    \ar{r}{\check\delta}
                & 0
            \end{tikzcd}
        \end{equation}
        Taking the signs of the total differential into account, this gives us the closed element
        \begin{equation*}
            \left(
                0\,,\,\,
                \frac{-1}{10}\tr\big(A^5\big)\,,\,\,
                \rho(A,X)\,,\,\,
                \tr(AXY)\,,\,\,
                0\,,\,\,
                0\,,\,\,
                0
            \right)\in\Tot^6\cech^\bullet\big(\Omega_X^\anotherbullet\big)
        \end{equation*}
        where
        \begin{equation*}
            \rho(A,X) = \frac{-1}{4}\tr\big(AXAX\big) + \frac12 \tr\big(A^2X^2\big) + \frac{-1}{2}\tr\big(A^3X\big) + \frac{-1}{2}\tr\big(AX^3\big).
        \end{equation*}

    \subsection{The fourth Atiyah class and beyond}

        Looking at the first three Atiyah classes, there are some evident patterns: for example, the Čech $1$-cocycle always seems to be some multiple of $\tr(A^{2k-1})$.
        Beyond some vague pattern recognition, however, it becomes increasingly hard to work with the $k$-th Atiyah class for $k\geqslant4$, mainly due to the cumbersome number of monomials whose coefficients we have to equate.
        It seems likely that there are two patterns --- one for $k$ odd, and one for $k$ even --- but even this is hard to verify, given that $k=0$ and $k=1$ are both rather trivial, and $k\geqslant4$ is so unwieldy that it becomes harder to recognise any patterns ``by hand''.\footnote{Thanks to a comment by Mahmoud Zeinalian, we see that the coefficients $1$, $\frac13$, $\frac{1}{10}$, $\frac{1}{35}$, etc. of the degree $(1,p-1)$ terms (i.e. those that are simply the traces of odd powers of the Maurer-Cartan form $A = g^{-1}\mathrm{d}g$) are given by \[\frac{n!(n+1)!}{(2n+1)!} = \frac{\Gamma(n+1)\Gamma(n+2)}{\Gamma(2n+2)} = (2n+2)\mathrm{B}(n+1,n+2),\] where $\mathrm{B}(x,y)$ is the beta function (or Euler integral of the first kind). It seems possible that the other coefficients could be expressed in terms of ``iterated beta functions'', using the formulas coming from the fibre integration approach we study in \cref{section:simplicial-construction-for-vector-bundles}.}
        For the sake of completeness, we give below the lift of $\expat{4}_E$, calculated by a basic Haskell implementation of the naive method used so far: calculating all non-commutative monomials of a certain degree, applying the Čech differential, and then equating coefficients.

        The lift of $\expat{4}_E$ in the total complex is
        \begin{equation*}
            \left(
                0\,,\,\,
                \frac{-1}{35}\tr P^{(1,7)}\,,\,\,
                \frac15\tr P^{(2,6)}\,,\,\,
                \frac15\tr P^{(3,5)}\,,\,\,
                \tr P^{(4,4)}\,,\,\,
                0\,,\,\,
                0\,,\,\,
                0\,,\,\,
                0
            \right) \in \Tot^8\cech^\bullet\big(\Omega^\anotherbullet_X\big)
        \end{equation*}
        where
        \begin{align*}
            P^{(4,4)}
            &= A(B-A)(C-B)(D-C)
        \\  P^{(3,5)}
            &=
            \frac{13}{5} A^5
            +13 A^4(B-A)
            +5 A^3(B-A)^2
            +5 A^3(B-A)(C-A)
        \\  &\quad+3 A^3(C-A)(B-A)
            +4 A^2(B-A)A(B-A)
            +4 A^2(B-A)A(C-A)
        \\  &\quad+3 A^2(B-A)^3
            - A^2(B-A)^2(C-A)
            +5A^2(B-A)(C-A)^2
        \\  &\quad+5A^2(C-A)A(B-A)
            +2 A^2(C-A)(B-A)^2
            + A^2(C-A)(B-A)(C-A)
        \\  &\quad+3 A^2(C-A)^2(B-A)
            - A(B-A)A(C-A)(B-A)
            +5 A(B-A)A(C-A)^2
        \\  &\quad-5 A(B-A)^2(C-A)(B-A)
            +5 A(B-A)(C-A)A(C-A)
            +5 A(B-A)(C-A)^3
        \\  &\quad+4 \big(A(C-A)\big)^2(B-A)
            -2 A(C-A)(B-A)^3
            +4 A(C-A)(B-A)^2(C-A)
        \\  &\quad+ A\big((C-A)(B-A)\big)^2
            +2 A(C-A)^2(B-A)^2
            + A(C-A)^2(B-A)(C-A)
        \\  &\quad+3 A(C-A)^3(B-A)
        \\  P^{(2,6)}
            &=
            5 A^5(B-A)
            -4 A^4(B-A)^2
            + A^3(B-A)A(B-A)
            + A^3(B-A)^3
        \\  &\quad-5 A^2(B-A)A(B-A)^2
            -4 A^2(B-A)^2A(B-A)
            -4 A^2(B-A)^4
        \\  &\quad+\frac{1}{3} \big(A(B-A)\big)^3
            + A(B-A)A(B-A)^3
            + A(B-A)^5
        \\  P^{(1,7)}
            &= A^7
        \end{align*}

        The moral of the above calculation is the following: we need to find a better way of doing this.

\section{Characteristic classes via admissible simplicial connections}\label{section:characteristic-classes-via-admissible-simplicial-connections}

    \begin{remark}\label{remark:the-motivating-example-recalled}
        We recall, from \cite[§4.2]{Hosgood2021}, the motivating example for the definition of {admissibility}.

        Write $\mathcal{C}$ to mean the category whose objects are pairs $(V,\varphi)$ of finite-dimensional vector spaces $V$ along with an endomorphism $\varphi$, and whose morphisms $f\colon(V,\varphi)\to(W,\psi)$ are the morphisms $f\colon V\to W$ such that $f\circ\varphi=\psi\circ f$.
        We say that $(V,\varphi)$ is \define{flat} if $\varphi=0$.
        Then we consider the Grothendieck group $K(\mathcal{C})$, and introduce the equivalence relation $[(V,0)]\sim[(0,0)]$ to obtain $K(\mathcal{C})/\!\!\sim$.
        Letting $E\colon\mathcal{C}\to\mathcal{C}$ be the endofunctor that sends $(V,\varphi)$ to $(V/\Ker\varphi,0)$, we write $\LL_E\mathcal{C}$ to mean the localisation of $\mathcal{C}$ at all morphisms that become isomorphisms after applying $E$ (the wide subcategory of which we denote by $\mathcal{W}$).

        Then a morphism $f\colon(V,\varphi)\to(W,\psi)$ is in $\mathcal{W}$ if and only if it is \define{admissible}: if and only if there exist sub-spaces $V_1\hookrightarrow V$ and $W_1\hookrightarrow W$ such that
            \begin{enumerate}
                \item $V_1\subseteq\Ker\varphi$; $W_1\subseteq\Ker\psi$;
                \item $f$ restricts to a morphism $V_1\to W_1$; and
                \item $f$ descends to an isomorphism $V/V_1\congto W/W_1$.
            \end{enumerate}
        In fact, there is a canonical isomorphism $K(\LL_E\mathcal{C})\congto K(\mathcal{C})/\!\!\sim$.
    \end{remark}

    \subsection{Generalised invariant polynomials}

        \begin{definition}\label{definition:generalised-invariant-polynomial}
            Using the notation of \cref{remark:the-motivating-example-recalled}, a sequence $(P_k)_{k\in\mathbb{N}}$ of $\mathbb{C}$-valued polynomials $P_k$ on $\mathcal{C}^{\otimes n}$ is said to be a \define{generalised invariant polynomial of degree $n$} if the following conditions are satisfied:
            \begin{enumerate}[(i)]
                \item each $P_k$ is a $\GL_k$-invariant $\mathbb{C}$-valued $\oplus$-additive polynomial of degree $n$ on
                \[
                    \left\{\bigotimes_{i=1}^n(V_i,\varphi_i)\in\mathcal{C}^{\otimes n} \Bigm\vert \dim V_i=k\right\}
                    \subset \mathcal{C}^{\otimes n};
                \]
                \item the $P_k$ each factor through $K(\mathcal{C})/\!\sim$.
            \end{enumerate}
            The second condition is equivalent to the following: the $P_k$ all satisfy the \define{extension by zero} property: $P_k=P_{k+1}\circ\iota_k$ for all $k\in\mathbb{N}$, where $\iota_k$ is the linear embedding $\mathbb{C}^k\hookrightarrow\mathbb{C}^{k+1}$ corresponding to $v\mapsto\smqty(v\\0)$.
            This extension by zero property, when combined with additivity, tells us that $P$ is \define{fully additive}:
            \[
                P_{\dim(V+W)}\big((V,\varphi)\oplus(W,\psi)\big) = P_{\dim V}\big((V,\varphi)\big) + P_{\dim W}\big((W,\psi)\big).
            \]

            Mirroring classical notation, given some \mbox{$P_\bullet=(P_k)_{k\in\mathbb{N}}$} on $\mathcal{C}^{\otimes n}$, we write $\widetilde{P_\bullet}$ to mean the sequence of invariant and additive polynomials on $\mathcal{C}$ given by
            \[
                \widetilde{P}_{\dim V}\big((V,\varphi)\big) = P_{\dim V}\left((V,\varphi)^{\otimes n}\right).
            \]
        \end{definition}

        \begin{example}
            The prototypical example of a generalised invariant polynomial is the degree-$1$ polynomial given by $P_k=\tr$ for all $r\in\mathbb{N}$, where the trace is of the endomorphism part of a pair;
            or, more generally, the degree $n$ polynomial given by $P_k^n=\tr\circ\mu^n$, where $\mu^n$ is the multiplication map that sends an $n$-fold tensor product of endomorphisms to the endomorphism given by the composition of all the endomorphisms.\footnote{The fact that this is well defined relies on being able to view all the $V_i$ in the tensor product as being identical, but we can do this thanks to the $\GL$-invariance of the trace, as well as the extension by zero property.}
        \end{example}

        \begin{lemma}\label{lemma:generalised-invariant-polynomials-are-admissible-invariant}
            Using the notation from \cref{remark:the-motivating-example-recalled}, let $f\colon(V,\varphi)\to(W,\psi)$ be an admissible morphism, and $(P_k)_{k\in\mathbb{N}}$ a generalised invariant polynomial of degree $n$ that is zero on an $n$-fold tensor product of flat objects.
            Then $\widetilde{P_k}\big((V,\varphi)\big)=\widetilde{P_k}\big((W,\psi)\big)$ for all $k\in\mathbb{N}$.
            \begin{proof}
                By definition, there exist decompositions $V\cong V_1\oplus V_2$ and $W\cong W_1\oplus W_2$ such that $\varphi(V_1)=\psi(W_1)=0$ and $f$ restricts to an isomorphism $V/V_1\congto W/W_1$.
                Then, necessarily, $\dim V_2=\dim W_2=s$ for some $s\in\mathbb{N}$.
                Since $P_k$ is fully additive,
                \begin{equation*}
                    \widetilde{P}_{\dim V}\big((V,\varphi)\big) = \widetilde{P}_{\dim V}\big((V_1,0)\oplus(V_2,\varphi\restricted V_2)\big) = \widetilde{P}_{\dim V_1}\big((V_1,0)\big) + \widetilde{P}_{s}\big((V_2,\varphi\restricted V_2)\big).
                \end{equation*}
                But $P_k$ is assumed to be zero on an $n$-fold tensor product of flat objects, and so
                \begin{equation*}
                    \widetilde{P}_{\dim V_1}\big((V_1,0)\big) + \widetilde{P}_{s}\big((V_2,\varphi\restricted V_2)\big) = \widetilde{P}_{s}\big((V_2,\varphi\restricted V_2)\big).
                \end{equation*}
                Then $\GL_k$-invariance tells us that
                \begin{equation*}
                    \widetilde{P}_{s}\big((V_2,\varphi\restricted V_2)\big) = \widetilde{P}_{s}\big((W_2,\varphi\restricted W_2)\big),
                \end{equation*}
                and this is equal to $\widetilde{P}_{\dim W}\big((W,\psi)\big)$, again by full additivity and being zero on flat objects.
            \end{proof}
        \end{lemma}

        \begin{corollary}\label{corollary:invariant-polynomial-of-admissible-form-is-simplicial}
            Let $\omega_\bullet$ be an admissible endomorphism-valued simplicial $r$-form\footnote{See \cite[Definition~4.13]{Hosgood2021}.} on $\mathcal{E}^\bullet$, and $P_\anotherbullet=(P_k)_{k\in\mathbb{N}}$ a generalised invariant polynomial of degree $n$ that is zero on $n$-fold tensor products of flat objects.
            Then $\widetilde{P_\anotherbullet}(\omega_\bullet)$ is a simplicial $rn$-form.
            \begin{proof}
                First of all, we need to verify that $\widetilde{P_\anotherbullet}(\omega_\bullet)$ is well defined.
                Over any $U_{\alpha_0\ldots\alpha_p}$, we have the pair $(\overline{\mathcal{E}_p},\omega_p)\restricted U_{\alpha_0\ldots\alpha_p}\times\Delta^p$ consisting (after trivialisation) of an $\OO_{U_{\alpha_0\ldots\alpha_p}\times\Delta^p}$-module of rank $\mathfrak{r}(\alpha_0\ldots\alpha_p)$ along with a form-valued endomorphism.
                As mentioned in \cite[Remark~2.3]{Hosgood2021}, the rank $\mathfrak{r}(p)=\mathfrak{r}(\alpha_0\ldots\alpha_p)$ is independent of the open set $U_{\alpha_0\ldots\alpha_p}$.
                This means that we can define $\widetilde{P}_{\mathfrak{r}(p)}(\omega_p)$ by applying $\widetilde{P}_{\mathfrak{r}(p)}$ to the endomorphism part of $\omega_p^{\otimes n}$, and wedging the $n$-fold tensor product of the form part.

                Then we need to show that $(\nerve{\bullet} f_p^i\times\id)^*\widetilde{P}_{\mathfrak{r}(p-1)}(\omega_{p-1}) = (\id\times f_p^i)^*\widetilde{P}_{\mathfrak{r}(p)}(\omega_p)$.
                Now we can use the extension by zero property to replace $\widetilde{P}_{\mathfrak{r}(p-1)}$ and $\widetilde{P}_{\mathfrak{r}(p)}$ with $\widetilde{P}_{\mathfrak{s}}$, where $\mathfrak{s}=\max\{\mathfrak{r}(p-1),\mathfrak{r}(p)\}$.
                But then, by admissibility and \cref{lemma:generalised-invariant-polynomials-are-admissible-invariant}, it suffices to show that $\widetilde{P}_\mathfrak{s}$ commutes with pullbacks.
                Because $\widetilde{P}_{\mathfrak{s}}$ is just given by the wedge product of forms on the form part of $\omega_p$, it commutes with pullbacks there.
                On the endomorphism part of $\omega_p$, since we are working locally, the pullback is simply a change of basis, and so commutes with the $\widetilde{P}_{\mathfrak{s}}$ by $\GL$-invariance.
            \end{proof}
        \end{corollary}

    \subsection{Simplicial Chern-Weil}\label{subsection:simplicial-chern-weil}

        The simplicial version of Chern-Weil theory that we will now discuss is almost identical to the classical one, such as that found in \cite[§4.4]{Huybrechts2005}: we can apply the classical arguments ``simplicial level by simplicial level'', and then use the admissibility condition to ensure that these all glue.

        \begin{lemma}\label{lemma:chern-weil-lemmas}
            Let $\mathcal{E}^\bullet$ be a vector bundle on the nerve, and $\nabla_\bullet$ an admissible simplicial connection on $\mathcal{E}^\bullet$.
            Let $A_\bullet$ be an admissible endomorphism-valued simplicial $1$-form, $\gamma_\bullet^j$ admissible endomorphism-valued simplicial $d_j$-forms, and let $P_\bullet$ be a generalised invariant polynomial of degree $n$.
            Then
            \begin{enumerate}
                \item $\widetilde{P}_\anotherbullet(\kappa(\nabla_\bullet + t A_\bullet)) = \widetilde{P}_\anotherbullet(\kappa(\nabla_\bullet)) + \kappa t P_\anotherbullet(\kappa(\nabla_\bullet),\ldots,\kappa(\nabla_\bullet),\nabla_\bullet(A_\bullet)) + \OO(t^2)$;
                \item $\kappa(\nabla_\bullet + A_\bullet) = \kappa(\nabla_\bullet) + A_\bullet\cdot A_\bullet + \nabla_\bullet(A_\bullet)$;
                \item $\nabla_\bullet(\kappa(\nabla_\bullet)) = 0$;
                \item $\d P_\anotherbullet(\gamma_\bullet^1,\ldots,\gamma_\bullet^k) = \sum_{j=1}^k (-1)^{\sum_{\ell=1}^{j-1}d_j} P_\anotherbullet(\gamma_\bullet^1,\ldots,\nabla_\bullet(\gamma_\bullet^j),\ldots,\gamma_\bullet^k)$;
                \item $\d P_\anotherbullet(\kappa(\nabla_\bullet)) = 0$.
            \end{enumerate}

            \begin{proof}
                The first statement is found in the proof of \cite[Lemma~4.4.6]{Huybrechts2005}, and so we simply apply it to each $P_{\mathfrak{r}(p)}(\kappa(\nabla_p))$, where $\mathfrak{r}(p)$ is the rank of $\mathcal{E}^p$; the last four are \cite[Lemma~4.3.4,Lemma~4.3.5,Lemma~4.4.4,Corollary~4.4.5]{Huybrechts2005} (respectively) applied to each $P_{\mathfrak{r}(p)}(\kappa(\nabla_p))$.
            \end{proof}
        \end{lemma}

        \begin{lemma}\label{lemma:admissible-compatible-gives-exact}
            If $\nabla_\bullet$ and $\nabla'_\bullet$ are two \emph{admissible}\footnote{That is, a connection $\nabla_p$ on each $\mathcal{E}_p$ such that the curvatures define a simplicial differential form and such that we can actually evaluate a generalised invariant polynomial on the curvatures in a choice-free way. See \cite[Definition~4.21]{Hosgood2021} for the precise definition.} simplicial connections on $\mathcal{E}^\bullet$ such that their difference is admissible, and $P_\bullet$ is a generalised invariant polynomial of degree $n$, then
            \[
                [P_\anotherbullet(\kappa(\nabla_\bullet))]
                =
                [P_\anotherbullet(\kappa(\nabla'_\bullet))],
            \]
            i.e. the difference $P_\anotherbullet(\kappa(\nabla'_\bullet)) - P_\anotherbullet(\kappa(\nabla_\bullet)) \in \Omega_{\nerve{\bullet}}^{\Delta,2n}$ is exact in the simplicial de~Rham complex.

            \begin{proof}
                This is \cite[Lemma~4.4.6]{Huybrechts2005} applied in each simplicial degree (as in the proof of \cref{lemma:chern-weil-lemmas}), but where we need the admissibility condition to see that everything glues: writing (locally) $\nabla_\bullet = \d+A_\bullet$ and $\nabla'_\bullet = \d+A'_\bullet$, we see that
                \begin{align*}
                    \kappa(\nabla'_\bullet) - \kappa(\nabla_\bullet)
                    &= \d(A'_\bullet - A_\bullet) + A_\bullet\cdot A_\bullet + A'_\bullet\cdot A'_\bullet
                \\  &= \d(A'_\bullet - A_\bullet) + (A'_\bullet - A_\bullet)^2 + A'_\bullet\cdot A_\bullet + A_\bullet\cdot A'_\bullet.
                \end{align*}
                By skew-symmetry, the last two terms will cancel when we apply $P_\bullet$, and so we want the difference $B_\bullet = A'_\bullet - A_\bullet$ to be \emph{admissible}, since then $P_\anotherbullet(B_\bullet)$ will be an endomorphism-valued simplicial form (by \cref{corollary:invariant-polynomial-of-admissible-form-is-simplicial}), and so $P_\anotherbullet(\d B_\bullet) = \d P_\anotherbullet(B_\bullet)$ will be closed in the simplicial de~Rham complex.
                But this is the case, by hypothesis.
            \end{proof}
        \end{lemma}

        \begin{remark}\label{remark:gidz-on-green-is-compatible}
            Recall \cite[Lemma~4.31]{Hosgood2021}: generated-in-degree-zero simplicial connections on a Green (complex of) vector bundle(s) on the nerve form a \define{compatible family}: the difference of any two such connections is an \emph{admissible} endomorphism-valued simplicial form.
        \end{remark}

    \subsection{The Green example}

        \begin{definition}\label{definition:simplicial-atiyah-classes}
            Given a vector bundle on the nerve $\mathcal{E}^\bullet$ and an \emph{admissible} simplicial connection $\nabla_\bullet$ on $\mathcal{E}^\bullet$, we define the \define{$k$-th simplicial exponential Atiyah class} $\simpexpat{k}_{\mathcal{E}^\bullet}$ by
            \begin{equation*}
                \simpexpat{k}_{\mathcal{E}^\bullet}
                =
                \left\{
                    \epsilon_k\kappa\left(\nabla_\bullet\right)^k
                \right\}
                \in
                \sheafend(\mathcal{E}^\bullet)\otimes\Omega_{\nerve{\bullet}}^{2k,\Delta}
            \end{equation*}
            and the \define{$k$-th simplicial standard Atiyah class} $\simpstanat{k}_{\mathcal{E}^\bullet}$ by
            \begin{equation*}
                \simpstanat{k}_{\mathcal{E}^\bullet}
                =
                \left\{
                    \epsilon_k\kappa\left(\nabla_\bullet\right)^{\wedge k}
                \right\}
                \in
                \sheafend(\mathcal{E}^\bullet)\otimes\Omega_{\nerve{\bullet}}^{2k,\Delta}
            \end{equation*}
            where $\epsilon_k=(-1)^{k(k-1)/2}$.
        \end{definition}

        \begin{remark}
            Both $\tr\simpexpat{k}_{\mathcal{E}^\bullet}$ and $\tr\simpstanat{k}_{\mathcal{E}^\bullet}$ are generalised invariant polynomials evaluated on the curvature $\kappa(\nabla_\bullet)$.
            That is, \emph{taking the trace} of the simplicial Atiyah classes gives us characteristic classes in the simplicial de~Rham complex.
            The fact that this is well defined relies on the fact that the simplicial connection is admissible (indeed, this is exactly the purpose of admissibility, cf. \cite[Remark~4.24]{Hosgood2021}).
        \end{remark}

        \begin{example}
            We now finish our study of the example of Green that was discussed in \cite[Examples~2.23 and 4.16]{Hosgood2021}.
            Recall that we had $X=\mathbb{P}_\mathbb{C}^1\cong\mathbb{C}\cup\{\infty\}$ with the cover $\cover=\{U_1,U_2\}$, where $U_1=X\setminus\{\infty\}$ and $U_2=X\setminus\{0\}$, and the coherent sheaf $\mathscr{F}$ given by $\OO_X/\mathscr{I}$, where $\mathscr{I}=\mathbb{I}(\{0\})$ is the sheaf of ideals corresponding to the subvariety $\{0\}\subset X$.
            We had constructed the barycentric connections $\nabla_\bullet^i$ on $\mathcal{E}^{\bullet,i}$ as
            \begin{align*}
                \nabla_\bullet^0\text{ on }\mathcal{E}^{\bullet,0}\text{ is given by }
                &\begin{cases}
                    \nabla_0^0 = t_0\d&=\d\\
                    \nabla_1^0 = t_0\d+t_1\d&=\d
                \end{cases}\\
                \nabla_\bullet^1\text{ on }\mathcal{E}^{\bullet,1}\text{ is given by }
                &\begin{cases}
                    \nabla_0^1 = t_0\d&=\d\\
                    \nabla_1^1 = t_0\d+t_1\left(\d+\frac{\d z}{z}\right)&=\d+t_1\frac{\d z}{z}.
                \end{cases}
            \end{align*}

            We see that the only non-trivial part of the simplicial connections is in (simplicial) degree 1, over $U_{12}$, and so we calculate these curvatures:
            \begin{align*}
                \kappa(\nabla_1^0) &= 0;\\
                \kappa(\nabla_1^1) &= \d\left(t_1\frac{\d z}{z}\right)+\left(t_1\frac{\d z}{z}\right)^2\\
                &= \frac{\d z}{z}\otimes\d t_1.
            \end{align*}
            Recalling \cref{remark:fibre-integration-only-depends-on-certain-parts}, we know that the fibre integral is given by integrating the \mbox{$(k-p,p)$} term over the $p$-simplex (and here we are taking $k=1$), and so we find that the\footnote{The first exponential and first standard Atiyah classes agree, by definition, so we do not need to say which we are using.} Atiyah classes are given by
            \begin{align*}
                \int_{\Delta^1}\kappa(\nabla_1^0) &= 0;\\
                \int_{\Delta^1}\kappa(\nabla_1^1) &= \int_0^1\frac{\d z}{z}\d t_1 = \frac{\d z}{z}.
            \end{align*}

            Note that the square of either curvature is zero, and so only the \emph{first} Atiyah class is non-trivial --- all higher ones are zero.
            Using the convention/definition that the zero-th Atiyah class is $1$, we have the total Atiyah classes (that is, Chern characters), defined as the sums of all the Atiyah classes:
            \begin{align*}
                \mathrm{at}^\mathrm{tot}_{\mathcal{E}^{\bullet,0}} &= 1,\\
                \mathrm{at}^\mathrm{tot}_{\mathcal{E}^{\bullet,1}} &= 1+\frac{\d z}{z}.
            \end{align*}

            Finally, if we say that, for a resolution
            \begin{equation*}
                \mathcal{R}^\anotherbullet=(\mathcal{R}^{\bullet,n}\to\ldots\to\mathcal{R}^{\bullet,1})
                \xrightarrow{\sim}\mathcal{R}^{\bullet,0},
            \end{equation*}
            the total Atiyah class (or Chern character) of $\mathcal{R}^{\bullet,0}$ is given by the alternating sum of the total Atiyah classes (or Chern characters) of the $\mathcal{R}^{\bullet,i}$, then
            \begin{equation*}
                \mathrm{ch}(\mathscr{F})
                =
                \left[\mathrm{at}^\mathrm{tot}_{\mathscr{F}_\bullet}\right]
                =
                \left[\frac{\d z}{z}\right].
            \end{equation*}
            This agrees with what one might calculate using a short exact sequence (writing the skyscraper sheaf as a quotient) and the Whitney sum formula, but is stronger, since Green's method gives actual \emph{representatives} of the cohomology class as well.
        \end{example}

\section{Simplicial construction for global vector bundles}\label{section:simplicial-construction-for-vector-bundles}

    \subsection{The barycentric connection}

        Throughout this section,
        \begin{itemize}
            \item as always, let $E$ be a vector bundle on $X$ with local \emph{flat} connections $\nabla_\alpha$;
            \item write $\pi_p\colon\nerve{p}\times\Delta^p\to\nerve{p}$ to mean the projection map;
            \item define $E^\bullet$, a vector bundle on the nerve, by setting $E^p=(\nerve{p}\to X)^*E$;
            \item write $\overline{E^p}$ to mean the pullback of $E^p$ along $\pi_p$:
                \[
                    \overline{E^p}
                    =
                    \pi_p^* E^p
                    =
                    \pi_p^{-1} E^p \otimes_{\pi_p^{-1}\OO_{\nerve{p}}} \OO_{\nerve{p}\times\Delta^p}.
                \]
        \end{itemize}

        \begin{remark}
            The pullback to the nerve of a vector bundle $E$ is \emph{strongly cartesian}: for all $\varphi\colon[p]\to[q]$ in $\Delta$, the morphism
            \[
                E^\bullet\varphi\colon
                \left(\nerve{\bullet}\varphi\right)^* E^p
                \longrightarrow
                E^q
            \]
            is an isomorphism.\footnote{It is actually an identity map.}
            This means that, writing $\zeta_p^i\colon[0]\to[p]$ to mean the map of simplices that sends $0$ to $i$, we can identify $(\nerve{\bullet}\zeta_p^i)^* E^0$ with $E^p$, by using the isomorphism $E_\bullet\zeta_p^i$ between them.
            This lets us, in particular, think of any connection $\nabla_{\alpha_i}$ on $E^0$ as a connection on $E^p$, and we keep the same notation to denote both.
            Then the $\nabla_\alpha$-flat sections of $\overline{E^p}$ are exactly those of the form $\pi_p^*(s)$, where $s$ is some $\nabla_\alpha$-flat section of $E\restricted U_\alpha$.
            Note also that the morphism $\OO_{\Delta^p}\to\OO_{\nerve{p}\times\Delta^p}$ gives us an $\OO_{\Delta^p}$-action on $E^p$.
        \end{remark}

        \begin{definition}[{\cite[34]{Green1980}}]
            The \define{barycentric connection} $\nabla_\bullet^\mu$ on $\overline{E^\bullet}$ is defined as
            \[
                \nabla_p^\mu
                =
                \sum_{i=0}^p t_i\nabla_{\alpha_i}
                \colon
                \overline{E^p}
                \to
                \overline{E^p}\otimes\Omega_{\nerve{p}\times\Delta^p}^1
            \]
            which acts on a section $s\otimes f$ of $\overline{E^p}=\pi_p^{-1}E^p\otimes\OO_{\nerve{p}\times\Delta^p}$ over $U_{\alpha_0\ldots\alpha_p}\times\Delta^p$ by
            \begin{align*}
                \nabla_p^\mu(s\otimes f)
                = \sum_{i=0}^p t_i\nabla_{\alpha_i}(s\otimes f)
                &= \sum_i t_i\nabla_{\alpha_i}( f s\otimes1)
            \\  &= \sum_i t_i\left( f\pi_p^*\left(\nabla_{\alpha_i}(s)\right)+s\otimes1\otimes\d f\right)
            \\  &= \sum_i \pi_p^*\left(t_i f\otimes\nabla_{\alpha_i}(s)\right)+s\otimes1\otimes t_i\d f
            \end{align*}
        \end{definition}

        \begin{remark}
            Recalling \cref{definition:omega-alpha-beta}, we can rewrite the barycentric connection as
            \[
                \nabla_p^\mu
                =
                \nabla_{\alpha_0} + \sum_{i=1}^p t_i \left(
                    \nabla_{\alpha_i} - \nabla_{\alpha_0}
                \right)
                =
                \nabla_{\alpha_0} + \sum_{i=1}^p t_i \omega_{\alpha_0\alpha_i}.
            \]
        \end{remark}

        \begin{remark}
            By \cite[Theorem~4.29, Remark~4.30]{Hosgood2021}, we know that the barycentric connection is indeed an admissible simplicial connection.
        \end{remark}

        \begin{definition}
            The \define{curvature} of the barycentric connection is defined level-wise, and acts on a section $\sigma$ of $\overline{E^p}$ over $U_{\alpha_0\ldots\alpha_p}\times\Delta^p$ by
            \begin{align*}
                \kappa\left(\nabla_p^\mu\right)(\sigma)
                &= \left(\nabla_{\alpha_0}+\sum_{i=1}^p t_i\omega_{\alpha_0\alpha_i}\right)^2\left(\sigma\right)
            \\  &= \nabla_{\alpha_0}^2\left(\sigma\right) + \sum_{i=1}^p\left[\left(\nabla_{\alpha_0}\left(\sigma\right)\cdot t_i\omega_{\alpha_0\alpha_i}\right)+\sigma\otimes\d(t_i\omega_{\alpha_0\alpha_i})\right]
            \\  & \quad+\sum_{i=1}^p\left[t_i\omega_{\alpha_0\alpha_i}\cdot\nabla_{\alpha_0}\left(\sigma\right)\right] + \sum_{i,j=1}^p\sigma\otimes\left(t_jt_i\omega_{\alpha_0\alpha_j}\omega_{\alpha_0\alpha_i}\right).
            \end{align*}
        \end{definition}

        \begin{definition}
            Let $\overline{\omega}_\bullet$ be the endomorphism-valued simplicial form given by
            \[
                \overline{\omega}_p
                =
                \sum_{i=1}^p t_i\omega_{\alpha_0\alpha_i}
            \]
            where we identify $t_i$ with $t_i I_\mathfrak{r}$, where $I_\mathfrak{r}$ is the $(\mathfrak{r}\times\mathfrak{r})$ identity matrix.
        \end{definition}

        \begin{remark}\label{remark:expression-for-barycentric-connection}
            In the case where $\sigma = \sigma_k^{\alpha_0} = \pi_p^*(s_k^{\alpha_0})$, where $\{s_1^{\alpha_0},\ldots,s_\mathfrak{r}^{\alpha_0}\}$ is a $\nabla_{\alpha_0}$-flat basis of $E\restricted U_{\alpha_0}$, we have that
            \[
                \nabla_{\alpha_0}\left(\sigma_k^{\alpha_0}\right)
                =
                \pi_p^*\left(\nabla_{\alpha_0}\left(s^{\alpha_0}_k\right)\right)
                =
                0.
            \]
            This lets us simplify the expression for the barycentric curvature, since
            \begin{align*}
                & \nabla_{\alpha_0}^2\left(\sigma_k^{\alpha_0}\right) + \sum_{i=1}^p\left[\left(\nabla_{\alpha_0}\left(\sigma_k^{\alpha_0}\right)\cdot t_i\omega_{\alpha_0\alpha_i}\right)+\sigma_k^{\alpha_0}\otimes\d(t_i\omega_{\alpha_0\alpha_i})\right]
            \\  & +\sum_{i=1}^p\left[t_i\omega_{\alpha_0\alpha_i}\cdot\nabla_{\alpha_0}\left(\sigma_k^{\alpha_0}\right)\right] + \sum_{i,j=1}^p\sigma_k^{\alpha_0}\otimes\left(t_jt_i\omega_{\alpha_0\alpha_j}\omega_{\alpha_0\alpha_i}\right)
            \\  =& \,\,\sum_{i=1}^p\sigma_k^{\alpha_0}\otimes\d(t_i\omega_{\alpha_0\alpha_i}) + \sum_{i,j=1}^p\sigma_k^{\alpha_0}\otimes\left(t_jt_i\omega_{\alpha_0\alpha_j}\omega_{\alpha_0\alpha_i}\right).
            \end{align*}

            In essence, this tells us that, given a (flat) basis of sections of $\overline{E^p} \restricted (U_{\alpha_0\ldots\alpha_p}\times\Delta^p)$ in the $U_{\alpha_0}$ trivialisation, the barycentric curvature acts simply as
            \[
                \kappa\left(\nabla_p^\mu\right)
                =
                \d\overline{\omega}_p + \overline{\omega}_p\cdot\overline{\omega}_p
            \]
            where $\d\overline{\omega}_p = (\d_X - \d_{\Delta^p})\overline{\omega}_p$ (following the Koszul convention, as explained in \cref{subsection:fibre-integration}).
        \end{remark}

        \begin{remark}\label{remark:barycentric-connection-is-admissible}
            Since $E^\bullet$ is the pullback of a global vector bundle, it is in particular strongly cartesian, and in fact Green.
            Since the barycentric connection is defined exactly as a connection generated in degree zero, we can apply \cite[Theorem~4.29]{Hosgood2021}, which tells us that it is an admissible simplicial connection.
            Further, \cref{remark:gidz-on-green-is-compatible} tells us that barycentric connections form a compatible family, and so \cref{lemma:admissible-compatible-gives-exact} says that the resulting characteristic classes are well defined.

            The barycentric curvature is, by definition (since the barycentric connection is admissible), admissible, but we can actually show that $\overline{\omega}_\bullet$ itself is an admissible endomorphism-valued simplicial 2-form, by an argument very similar to that of the proof of \cite[Theorem~4.29]{Hosgood2021}.
        \end{remark}

        \begin{lemma}
            The trace of the $k$-th simplicial (standard or exponential) Atiyah class is closed in the simplicial de~Rham complex.
            \begin{proof}
                We have already formally proven this in \cref{subsection:simplicial-chern-weil}, but we give here an explicit proof for the case $k=1$ as well.
                For any $p\in\mathbb{N}$,
                \begin{align*}
                    \d\left(\d\overline{\omega}_p+\overline{\omega}_p\cdot\overline{\omega}_p\right) &= \d^2\left(\sum_{i=0}^p t_i\omega_{\alpha_0\alpha_i}\right)+\d\left(\sum_{i,j=1}^p t_jt_i\omega_{\alpha_0\alpha_j}\omega_{\alpha_0\alpha_i}\right)\\
                    &=\d\left(\sum_{i,j=1}^p t_jt_i\omega_{\alpha_0\alpha_j}\omega_{\alpha_0\alpha_i}\right)\\
                    &=\sum_{i,j=1}^p\left[\omega_{\alpha_0\alpha_j}\omega_{\alpha_0\alpha_i}\otimes\d_{\Delta}(t_jt_i) + t_jt_i\d_X\left(\omega_{\alpha_0\alpha_j}\omega_{\alpha_0\alpha_i}\right)\right]\\
                    &=\sum_{i,j=1}^p\big[\omega_{\alpha_0\alpha_j}\omega_{\alpha_0\alpha_i}\otimes t_i\d t_j + \omega_{\alpha_0\alpha_j}\omega_{\alpha_0\alpha_i}\otimes t_j\d t_i\\
                    &\qquad\quad + t_jt_i\d_X\left(\omega_{\alpha_0\alpha_j}\right)\omega_{\alpha_0\alpha_i} - t_jt_i\omega_{\alpha_0\alpha_j}\d_X\left(\omega_{\alpha_0\alpha_i}\right)\big].
                \end{align*}
                But, recalling the proof of \cref{lemma:d-tr-omega-is-zero}, we know that
                \[
                    \tr\left(\omega_{\alpha_0\alpha_j}\omega_{\alpha_0\alpha_i}\right)=-\tr\left(\omega_{\alpha_0\alpha_i}\omega_{\alpha_0\alpha_j}\right),
                \]
                and \cref{lemma:d-omega-is-minus-omega-squared} tells us that $\d\omega_{\alpha_0\alpha_i}=-\omega_{\alpha_0\alpha_i}^2$, whence
                \begin{align*}
                    \tr\d\left(\d\overline{\omega}_p+\overline{\omega}_p\cdot\overline{\omega}_p\right) &= \tr\sum_{i,j=1}^p\big[\omega_{\alpha_0\alpha_j}\omega_{\alpha_0\alpha_i}\otimes t_i\d t_j + \omega_{\alpha_0\alpha_j}\omega_{\alpha_0\alpha_i}\otimes t_j\d t_i\\
                    &\qquad\qquad + t_jt_i\d_X\left(\omega_{\alpha_0\alpha_j}\right)\omega_{\alpha_0\alpha_i} - t_jt_i\omega_{\alpha_0\alpha_j}\d_X\left(\omega_{\alpha_0\alpha_i}\right)\big]\\
                    &= \tr\sum_{i,j=1}^p\big[\omega_{\alpha_0\alpha_j}\omega_{\alpha_0\alpha_i}\otimes t_i\d t_j - \omega_{\alpha_0\alpha_i}\omega_{\alpha_0\alpha_j}\otimes t_j\d t_i\\
                    &\qquad\qquad + t_jt_i\omega_{\alpha_0\alpha_j}^2\omega_{\alpha_0\alpha_i} - t_jt_i\omega_{\alpha_0\alpha_j}\omega_{\alpha_0\alpha_i}^2\big].
                \end{align*}
                For fixed $i,j$, the first two terms both change sign under $i\leftrightarrow j$, whence they contribute zero to the trace, since they are equal when $i=j$.
                We also know\footnote{The trace on endomorphism-valued simplicial forms is cyclic, up to a sign, as mentioned in the proof of \cref{lemma:d-tr-omega-is-zero}.} that $\tr\left(\omega_{\alpha_0\alpha_j}^2\omega_{\alpha_0\alpha_i}\right)=\tr\left(\omega_{\alpha_0\alpha_i}\omega_{\alpha_0\alpha_j}^2\right)$, whence the last two terms also both change sign under $i\leftrightarrow j$ and are equal when $i=j$, as above.
                Thus the trace is zero.
            \end{proof}
        \end{lemma}

    \subsection{The first simplicial Atiyah class}

        The first simplicial (exponential) Atiyah class is given by
        \begin{equation}
            \simpexpat{1}_E = \Bigg\{-\sum_{i=1}^p\omega_{\alpha_0\alpha_i}\otimes\d t_i - \sum_{i=1}^pt_i\omega_{\alpha_0\alpha_i}^2 + \sum_{i,j=1}^p t_jt_i\omega_{\alpha_0\alpha_j}\omega_{\alpha_0\alpha_i}\Bigg\}_{p\in\mathbb{N}}
        \end{equation}
        by using the barycentric connection, and \cref{remark:expression-for-barycentric-connection}.

        As explained in \cref{remark:fibre-integration-only-depends-on-certain-parts}, the fibre integral of $\simpexpat{1}_E$ depends only on the $(2,0)$, $(1,1)$, and $(0,2)$ parts, but here there is no $(0,2)$ part (i.e. there is no $\d t_j\d t_i$ term), and so
        \begin{align}
            \int_{\Delta^\bullet}\simpexpat{1}_E &= (-1)\int_{\Delta^1}-\sum_{i=1}^{p=1}\omega_{\alpha_0\alpha_i}\otimes\d t_i - \int_{\Delta^0}\sum_{i=1}^{p=0}t_i\omega_{\alpha_0\alpha_i}^2 + \int_{\Delta^0}\sum_{i,j=1}^{p=0} t_jt_i\omega_{\alpha_0\alpha_j}\omega_{\alpha_0\alpha_i}
        \nonumber
        \\  &= \int_{\Delta^1}\omega_{\alpha_0\alpha_1}\otimes\d t_1.
        \end{align}

        Continuing the calculation gives us the result that we expect: as in \cref{subsection:first-manual-atiyah-class}, we get that
        \begin{equation}
            \tr\int_{\Delta^\bullet}\simpexpat{1}_E = \tr\int_0^1\omega_{\alpha_0\alpha_1} \d t_1 = \underbrace{\tr(\omega_{\alpha_0\alpha_1})}_{p=1} = \tr\expat{1}_E,
        \end{equation}
        where we write that this term lives in degree~$p=1$ to remind us that the result is a Čech $1$-cocycle.

        \begin{remark}
            In general, as in the manual construction, the fibre integral of the $k$-th simplicial (exponential) Atiyah class will have terms in $\cech^{k-i}\left(\Omega_X^{k+i}\right)$ for $i=0,\ldots,k-1$, and so we make up for our notational laziness (namely, using $+$ to mean $\oplus$) by labelling the terms with their Čech degree.

            Note that, although the $(2k,0)$ part of $\simpexpat{k}_E$ is, in general, non-zero, when we fibre integrate we look at it on the 0-simplex (again, by \cref{remark:fibre-integration-only-depends-on-certain-parts}), and there it \emph{is} zero, since all the sums are trivially zero.
        \end{remark}

    \subsection{The second simplicial Atiyah class}

        Not forgetting the sign $\epsilon_2=-1$ from \cref{definition:simplicial-atiyah-classes}, we have that
        \begin{equation}
            \simpexpat{2}_E = \Bigg\{-\Bigg(-\sum_{i=1}^p\omega_{\alpha_0\alpha_i}\otimes\d t_i - \sum_{i=1}^p t_i\omega_{\alpha_0\alpha_i}^2 + \sum_{i,j=1}^p t_j t_i\omega_{\alpha_0\alpha_j}\omega_{\alpha_0\alpha_i}\Bigg)^2\Bigg\}_{p\in\mathbb{N}}
        \end{equation}
        but we also know that the only parts that will be non-zero after fibre integration are the $(2,2)$ parts on the $2$-simplex, and the $(3,1)$ parts on the $1$-simplex.

        The only $(2,2)$ part comes from the first half of the $(\d\overline{\omega}_2)^2$ term, which gives us
        \begin{align}
            \tr\int_{\Delta^2}\simpexpat{2}_E
            &= \tr(-1)^{2\cdot2}\int_{\Delta^2}-\left(-\sum_{i=1}^2\omega_{\alpha_0\alpha_i}\otimes\d t_i\right)^2
        \nonumber
        \\  &= \tr\int_{\Delta^2}-\left(\sum_{i,j=1}^2(\omega_{\alpha_0\alpha_j}\otimes\d t_j)\cdot(\omega_{\alpha_0\alpha_i}\otimes\d t_i)\right)
        \nonumber
        \\  &= \tr\int_{\Delta^2}\sum_{i,j=1}^2\omega_{\alpha_0\alpha_j}\omega_{\alpha_0\alpha_i}\otimes\d t_j\d t_i
        \nonumber
        \\  &= \tr\int_{\Delta^2}\Big(\omega_{\alpha_0\alpha_1}^2\otimes(\d t_1)^2 + \omega_{\alpha_0\alpha_1}\omega_{\alpha_0\alpha_2}\otimes\d t_1\d t_2
        \nonumber
        \\  &\qquad\quad+ \omega_{\alpha_0\alpha_2}\omega_{\alpha_0\alpha_1}\otimes\d t_2\d t_1 + \omega_{\alpha_0\alpha_2}^2\otimes(\d t_2)^2\Big)
        \nonumber
        \\  &= \tr\int_{\Delta^2} \Big(\omega_{\alpha_0\alpha_1}\omega_{\alpha_0\alpha_2} - \omega_{\alpha_0\alpha_2}\omega_{\alpha_0\alpha_1}\Big)\otimes\d t_1\d t_2
        \nonumber
        \\  &= \tr\int_0^1\int_0^{1-t_2} \Big(\omega_{\alpha_0\alpha_1}\omega_{\alpha_0\alpha_2} - \omega_{\alpha_0\alpha_2}\omega_{\alpha_0\alpha_1}\Big)\otimes\d t_1\d t_2
        \nonumber
        \\  &= \frac12\tr\Big(\omega_{\alpha_0\alpha_1}\omega_{\alpha_0\alpha_2} - \omega_{\alpha_0\alpha_2}\omega_{\alpha_0\alpha_1}\Big)
        \nonumber
        \\  &= \frac{1}{2}\cdot2\cdot\tr\omega_{\alpha_0\alpha_1}\omega_{\alpha_0\alpha_2}
        \nonumber
        \\  &= \tr\omega_{\alpha_0\alpha_1}(\omega_{\alpha_0\alpha_1}+\omega_{\alpha_1\alpha_2})
        \nonumber
        \\  &= \tr\omega_{\alpha_0\alpha_1}\omega_{\alpha_1\alpha_2}.
        \end{align}
        This means that, so far, we have
        \begin{equation*}
            \tr\int_{\Delta^\bullet}\simpexpat{2}_E = \underbrace{?}_{p=1} + \underbrace{\tr(\omega_{\alpha_0\alpha_1}\omega_{\alpha_1\alpha_2})}_{p=2}.
        \end{equation*}

        For the $(3,1)$ part, we work on the 1-simplex and get
        \begin{align*}
            \tr\int_{\Delta^1}\simpexpat{2}_E
            &= \tr(-1)^{3\cdot1}\int_{\Delta^1}-\,\Bigg(-\sum_{i,j=1}^1 (\omega_{\alpha_0\alpha_j}\otimes\d t_j)\cdot(-t_i\omega_{\alpha_0\alpha_i}^2)
        \\  &\qquad\qquad\qquad\qquad- \sum_{i,j=1}^1 (-t_j\omega_{\alpha_0\alpha_j}^2)\cdot(\omega_{\alpha_0\alpha_i}\otimes\d t_i)
        \\  &\qquad\qquad\qquad\qquad- \sum_{i,j,k=1}^1 (\omega_{\alpha_0\alpha_k}\otimes\d t_k)\cdot(t_j t_i\omega_{\alpha_0\alpha_j}\omega_{\alpha_0\alpha_i})
        \\  &\qquad\qquad\qquad\qquad- \sum_{i,j,k=1}^1 (t_k t_j\omega_{\alpha_0\alpha_k}\omega_{\alpha_0\alpha_j})\cdot(\omega_{\alpha_0\alpha_i}\otimes\d t_i)\Bigg)
        \\  &= \tr\int_0^1 2\omega_{\alpha_0\alpha_1}^3(t_1-t_1^2)\d t_1
        \\  &= \frac13\tr\omega_{\alpha_0\alpha_1}^3.
        \end{align*}

        Finally then, we have
        \begin{equation}
            \tr\int_{\Delta^\bullet}\simpexpat{2}_E = \underbrace{\frac13\tr\omega_{\alpha_0\alpha_1}^3}_{p=1} + \underbrace{\tr(\omega_{\alpha_0\alpha_1}\omega_{\alpha_1\alpha_2})}_{p=2} = \tr\expat{2}_E
        \end{equation}
        which again agrees with the result of \cref{subsection:second-manual-atiyah-class}.

    \subsection{The third simplicial Atiyah class}\label{subsection:third-simplicial-atiyah-class}

        \begin{remark}
            There is a subtlety in the calculations when we reach the third simplicial Atiyah class due to our choice of conventions for Čech cocycles: we don't assume skew-symmetry of cocycles (i.e. that exchanging two indices changes sign), but it is true that skew-symmetrisation of cocycles is a quasi-isomorphism, and so doesn't change the resulting cohomology class.
            If we had worked with skew-symmetric Čech cocycles from the start then this calculation would seem somewhat simpler.

            In particular, the first two (exponential) Atiyah classes agree with those that we manually constructed in \cref{subsection:first-manual-atiyah-class,subsection:second-manual-atiyah-class} on the nose, whereas for $k\geqslant3$ we will have equality only in cohomology.
        \end{remark}

        \begin{definition}
            We write $\ss{p}$ to mean the \define{skew-symmetrisation} of a Čech $p$-cochain, so that
            \[
                (\ss{p}c)_{\alpha_0\ldots\alpha_p}
                =
                \frac{1}{(p+1)!}\sum_{\sigma\in\sym{p+1}}\sgn{\sigma}c_{\alpha_{\sigma(0)}\ldots\alpha_{\sigma(p)}}.
            \]
            where $\sgn{\sigma}$ denotes the sign of $\sigma$.
        \end{definition}

        So, we begin the calculation of $\simpexpat{3}_E$.
        Writing $\mu_i$ to mean $\omega_{\alpha_0\alpha_i}$, we start by calculating the $(3,3)$ part as follows:
        \begin{align}
            \tr\int_{\Delta^3}\simpexpat{3}_E
            &= \tr(-1)^{3\cdot3}\int_{\Delta^3}-\sum_{i,j,k=1}^3 (-\mu_k\otimes\d t_k)\cdot(-\mu_j\otimes\d t_j)\cdot(-\mu_i\otimes\d t_i)
        \nonumber
        \\  &= \tr\int_{\Delta^3}\sum_{i,j,k=1}^{p=3} \mu_k\mu_j\mu_i\otimes\d t_k\d t_j\d t_i
        \nonumber
        \\  &= \tr\int_{\Delta^3}\sum_{\sigma\in S_3}\sgn{\sigma}\,\mu_{\sigma(1)}\mu_{\sigma(2)}\mu_{\sigma(3)}\otimes\d t_1\d t_2\d t_3
        \nonumber
        \\  &= \frac16 \tr\sum_{\sigma\in\sym{3}}\sgn{\sigma}\,\mu_{\sigma(1)}\mu_{\sigma(2)}\mu_{\sigma(3)}
        \nonumber
        \\  &= \frac12\tr(\mu_1\mu_2\mu_3 - \mu_1\mu_3\mu_2)
        \nonumber
        \\  &= \frac12\tr\Big(\omega_{\alpha_0\alpha_1}(\omega_{\alpha_0\alpha_1}+\omega_{\alpha_1\alpha_2})(\omega_{\alpha_0\alpha_1}+\omega_{\alpha_1\alpha_2}+\omega_{\alpha_2\alpha_3})
        \nonumber
        \\  &\qquad\quad -\omega_{\alpha_0\alpha_1}(\omega_{\alpha_0\alpha_1}+\omega_{\alpha_1\alpha_2}+\omega_{\alpha_2\alpha_3})(\omega_{\alpha_0\alpha_1}+\omega_{\alpha_1\alpha_2})\Big)
        \nonumber
        \\  &= \frac12\tr\Big(\omega_{\alpha_0\alpha_1}\omega_{\alpha_1\alpha_2}\omega_{\alpha_2\alpha_3} - \omega_{\alpha_0\alpha_1}\omega_{\alpha_2\alpha_3}\omega_{\alpha_1\alpha_2}\Big).
        \end{align}
        But note that both of the terms in this last expression skew-symmetrise to the same thing, modulo a minus sign:
        \begin{equation}
            \sum_{\tau\in\sym{4}}\sgn{\tau}\, \omega_{\alpha_{\tau(0)}\alpha_{\tau(1)}}\omega_{\alpha_{\tau(1)}\alpha_{\tau(2)}}\omega_{\alpha_{\tau(2)}\alpha_{\tau(3)}} = -\sum_{\tau\in\sym{4}}\sgn{\tau}\, \omega_{\alpha_{\tau(0)}\alpha_{\tau(1)}}\omega_{\alpha_{\tau(2)}\alpha_{\tau(3)}}\omega_{\alpha_{\tau(1)}\alpha_{\tau(2)}}.
        \end{equation}
        Since skew-symmetrisation doesn't change the class in cohomology, we see that
        \begin{equation}
            \left[\tr\int_{\Delta^\bullet}\simpexpat{3}_E\right] = \big[\underbrace{?}_{p=1} + \underbrace{?}_{p=2} + \underbrace{\tr\omega_{\alpha_0\alpha_1}\omega_{\alpha_1\alpha_2}\omega_{\alpha_2\alpha_3}}_{p=3}\big].
        \end{equation}

        The $(4,2)$ part (including the sign $\epsilon_3=-1$) is given by
        \begin{equation}
            X^2Y + XYX + YX^2 - X^2Z - XZX - ZX^2
        \end{equation}
        where
        \begin{equation}
            X = -\sum_{i=1}^2\mu_i\otimes\d t_i\qquad
            Y = \sum_{i=1}^2 t_i\mu_i^2\qquad
            Z = \sum_{i,j=1}^2 t_j t_i\mu_j\mu_i.
        \end{equation}
        Using that $\int_{\Delta^2}\d t_1\d t_2=\int_0^1\int_0^{1-t_2}\d t_1\d t_2$, we calculate that
        \begin{enumerate}[(i)]
            \item $\int_{\Delta^2}t_1\d t_1\d t_2=\int_{\Delta^2}t_2\d t_1\d t_2=\frac16$;
            \item $\int_{\Delta^2}t_1^2\d t_1\d t_2=\int_{\Delta^2}t_2^2\d t_1\d t_2=\frac{1}{12}$;
            \item $\int_{\Delta^2}t_1t_2\d t_1\d t_2=\frac{1}{24}$.
        \end{enumerate}
        whence
        \begin{gather}
            \tr(-1)^{2\cdot2}\int_{\Delta^2}X^2(Y-Z) + X(Y-Z)X + (Y-Z)X^2
        \nonumber
        \\  = -\frac12\tr\mu_1^3\mu_2 - \frac12\tr\mu_1\mu_2^3 + \frac14\tr\mu_1\mu_2\mu_1\mu_2
        \nonumber
        \\  = -\frac14\tr\big((\omega_{\alpha_0\alpha_1}\omega_{\alpha_1\alpha_2})^2\big) - \frac12\tr\big(\omega_{\alpha_0\alpha_1}^3\omega_{\alpha_1\alpha_2} - \omega_{\alpha_0\alpha_1}\omega_{\alpha_1\alpha_2}^3\big).
        \end{gather}
        Comparing this to \cref{subsection:third-manual-atiyah-class}, we see that we have the same, except for a missing $+\frac12\tr\left(\omega_{\alpha_0\alpha_1}^2\omega_{\alpha_1\alpha_2}^2\right)$ term.
        But this missing term skew-symmetrises to zero, since it is invariant under the permutation that swaps $0$ and $2$.
        Thus the $(4,2)$ part is exactly what we wanted, and
        \begin{equation}
            \left[\tr\int_{\Delta^\bullet}\simpexpat{3}_E\right] = \Bigg[\underbrace{?}_{p=1}
            +\underbrace{%
                \begin{array}{c}
                    - \frac14\tr\big((\omega_{\alpha_0\alpha_1}\omega_{\alpha_1\alpha_2})^2\big)\\
                    - \frac12\tr\big(\omega_{\alpha_0\alpha_1}^3\omega_{\alpha_1\alpha_2}\big)\\
                    - \frac12\tr\big(\omega_{\alpha_0\alpha_1}\omega_{\alpha_1\alpha_2}^3\big)\\
                    + \frac12\tr\left(\omega_{\alpha_0\alpha_1}^2\omega_{\alpha_1\alpha_2}^2\right)
                \end{array}
            }_{p=2}
            + \underbrace{\tr\omega_{\alpha_0\alpha_1}\omega_{\alpha_1\alpha_2}\omega_{\alpha_2\alpha_3}}_{p=3}\Bigg].
        \end{equation}

        Finally, the $(5,1)$ part is
        \begin{equation}
            \tr(-1)^{5\cdot1}\int_{\Delta^1}\simpexpat{3}_E
            = \tr\int_0^1-(3t_1^2+ 3t_1^4 - 6t_1^3)\mu_1^5\otimes\d t_1
            = -\frac{1}{10}\omega_{\alpha_0\alpha_1}^5
        \end{equation}
        which agrees exactly with the manual construction from \cref{subsection:third-manual-atiyah-class}.
        Thus
        \begin{align}
            \left[\tr\int_{\Delta^\bullet}\simpexpat{3}_E\right]
            &=
            \Bigg[\underbrace{-\frac{1}{10}\omega_{\alpha_0\alpha_1}^5}_{p=1}
            +\underbrace{%
                \begin{array}{c}
                    - \frac14\tr\big((\omega_{\alpha_0\alpha_1}\omega_{\alpha_1\alpha_2})^2\big)\\
                    - \frac12\tr\big(\omega_{\alpha_0\alpha_1}^3\omega_{\alpha_1\alpha_2}\big)\\
                    - \frac12\tr\big(\omega_{\alpha_0\alpha_1}\omega_{\alpha_1\alpha_2}^3\big)\\
                    + \frac12\tr\left(\omega_{\alpha_0\alpha_1}^2\omega_{\alpha_1\alpha_2}^2\right)
                \end{array}
            }_{p=2}
            + \underbrace{\tr\omega_{\alpha_0\alpha_1}\omega_{\alpha_1\alpha_2}\omega_{\alpha_2\alpha_3}}_{p=3}\Bigg]
        \nonumber
        \\  &= \left[\tr\expat{3}_E\right].
        \end{align}

    \subsection{Agreement with the manual construction}

        \begin{theorem}\label{theorem:simplicial-atiyah-agrees-with-manual-atiyah}
            The degree-$(k,k)$ term in the trace of fibre integral of the $k$-th simplicial {exponential} Atiyah class agrees with the $k$-th {exponential} Atiyah class, up to skew-symmetrisation.
            That is,
            \[
                \ss{k}\left(
                    \tr\int_{\Delta^\bullet}\simpexpat{k}_E
                \right)^{(k,k)}\!\!\!\!\!\!
                =
                \ss{k}\tr\left(
                    \expat{k}_E
                \right)
                \in \cech^k_\cover\left(\Omega_X^k\right).
            \]
            \begin{proof}
                First we rewrite the left-hand side.
                Generalising the results of \cref{subsection:third-simplicial-atiyah-class}, we can write the term coming from fibre integration as
                \[
                    \frac{1}{k!} \sum_{\sigma\in\sym{k}}
                        \sgn{\sigma}\mu_{\sigma(1)} \cdots \mu_{\sigma(k)}
                    \overset{\ss{k}}{\longmapsto}
                    \frac{1}{k!(k+1)!} \sum_{\tau\in \sym{k+1}}
                        \sum_{\sigma\in\sym{k}}
                            \sgn{\tau\sigma}\omega_{\tau(0)\tau\sigma(1)} \cdots \omega_{\tau(0)\tau\sigma(k)}
                \]
                where $\sym{k}\leqslant\sym{k+1}$ acts on $\{0,1,\ldots,k\}$ by fixing $0$.
                But then, since $\sigma(0)=0$, we can rewrite this as
                \[
                    \frac{1}{k!(k+1)!} \sum_{\tau\in S_{k+1}}
                        \sum_{\sigma\in S_{k}}
                            \sgn{(\tau\sigma)}\omega_{\tau\sigma(0)\tau\sigma(1)} \cdots \omega_{\tau\sigma(0)\tau\sigma(k)}.
                \]
                Now we can use the fact that multiplication by an element of $\sym{k}\leqslant \sym{k+1}$ is an automorphism to perform a change of variables, giving us
                \[
                    \frac{1}{(k+1)!}\sum_{\eta\in S_{k+1}}\sgn{(\eta)}\,\omega_{\eta(0)\eta(1)}\cdots\omega_{\eta(0)\eta(k)}
                \]
                which is (trivially, since $\omega_{ii}=0$) equal to
                \[
                    \frac{1}{(k+1)!}\sum_{\eta\in S_{k+1}}\sgn{(\eta)}\,\prod_{i=1}^k\big(\omega_{\eta(0)\eta(i)}-\omega_{\eta(0)\eta(0)}\big).
                \]

                Next we rewrite the right-hand side.
                The skew-symmetrisation is simply
                \begin{align*}
                    \ss{k}\tr\left(
                        \expat{k}_E
                    \right)
                    &=
                    \frac{1}{(k+1)!} \sum_{\eta\in \sym{k+1}}
                        \sgn{\eta}\, \omega_{\eta(0)\eta(1)} \omega_{\eta(1)\eta(2)} \cdots \omega_{\eta(k-1)\eta(k)}
                \\  &= \frac{1}{(k+1)!} \sum_{\eta\in \sym{k+1}}
                        \sgn{\eta}\, \omega_{\eta(0)\eta(1)}(\omega_{\eta(0)\eta(2)}-\omega_{\eta(0)\eta(1)}) \cdots (\omega_{\eta(0)\eta(k)}-\omega_{\eta(0)\eta(k-1)})
                \\  &= \frac{1}{(k+1)!} \sum_{\eta\in \sym{k+1}}
                        \sgn{\eta}\,
                            \prod_{i=1}^k
                                \big(\omega_{\eta(0)\eta(i)}-\omega_{\eta(0)\eta(i-1)}\big).
                \end{align*}

                Now we prove equality.
                Since, \emph{for each fixed $\eta$}, there are no relations satisfied between the $\omega_{\eta(0)\eta(i)}$, we have two homogeneous degree~$k$ polynomials in $(k+1)$ free non-commutative variables.
                To emphasise the fact the following argument is purely abstract, we write $x_i=\omega_{0i}$ and define an action of $S_{k+1}$ on the $x_i$ by $x_{\eta(i)}=\omega_{\eta(0)\eta(i)}$.
                So showing that the left- and right-hand sides are equal amounts to showing that
                \[
                    A :=
                    \sum_{\eta\in \sym{k+1}}
                        \sgn{\eta}\,
                            \prod_{i=1}^k
                                \big(x_{\eta(i)}-x_{\eta(0)}\big)
                    =
                    \sum_{\eta\in \sym{k+1}}
                        \sgn{\eta}\,
                            \prod_{i=1}^k
                                \big(x_{\eta(i)}-x_{\eta(i-1)}\big)
                    =: B.
                \]

                Write $E$ to mean the $\mathbb{Q}$-linear span of degree-$k$ monomials in the $(k+1)$ free non-commutative variables $x_0,\ldots,x_k$, and let $\sigma_{p,q}\in S_{k+1}$ be the transposition that swaps $p$ with $q$.
                Then $\sigma_{p,q}$ gives an involution on $E$ by swapping $x_p$ with $x_q$, and thus $E\cong E_{p,q}(1)\oplus E_{p,q}(-1)$, where $E_{p,q}(\lambda)$ is the eigenspace corresponding to the eigenvalue $\lambda$.

                Let $H_{p,q}$ be the $\mathbb{Q}$-linear subspace of $E$ spanned by monomials that contain at least one $x_p$ or $x_q$.
                This subspace is stable\footnote{If a monomial $X$ contains, say, one $x_p$, then $\sigma_{p,q}X$ contains one $x_q$.} under $\sigma_{p,q}$, and so this space also splits as $H_{p,q}\cong H_{p,q}(1)\oplus H_{p,q}(-1)$.
                Further, we have the inclusion $H_{p,q}(-1)\subseteq E_{p,q}(-1)$.
                But if $X\in E_{p,q}(-1)$, then, in particular, $X$ must contain at least one\footnote{If not, then the action of $\sigma_{p,q}$ would be trivial and $X$ would lie in $E_{p,q}(1)$.} $x_p$ or $x_q$, so $X\in H_{p,q}$, whence $X\in H_{p,q}(-1)$.
                Thus $E_{p,q}(-1)=H_{p,q}(-1)$.

                The intersection $H$ of the $H_{p,q}$ over all \emph{distinct} pairs $(p,q)\in\{0,\ldots,k\}\times\{0,\ldots,k\}$ is the $\mathbb{Q}$-linear span of all monomials containing all but one of the $x_i$ (and, in particular, containing $k$ distinct $x_i$).
                But since $H_{p,q}(-1)=E_{p,q}(-1)$, we see that the intersection $E(-1)$ of all the $E_{p,q}(-1)$ is equal to $H(-1)$.
                Now both $A$ and $B$ are in $E_{p,q}(-1)$ for all $p,q$ (since the sign of $\omega_{p,q}$ is $-1$), and so $A,B\in E(-1)=H(-1)$.
                Since the coefficient of, for example, the $x_1\cdots x_k$ term is the same (and \emph{non-zero}\footnote{In fact, it is $1$, as seen by taking $\eta=\id$.}) in both $A$ and $B$ , it suffices to show that $H(-1)$ is one-dimensional.

                So let $X,Y\in H(-1)$ be monomials.
                Then each one contains $k$ distinct $x_i$, and so there exists some (unique) $\sigma\in S_{k+1}$ such that $\sigma X=\pm Y$.
                But, writing $\sigma=\sigma_{p_1,q_1}\cdots\sigma_{p_r,q_r}$, we know that $\sigma X=(-1)^r X$, whence $X=Y$, up to some sign (and so up to some scalar in $\mathbb{Z}$).
            \end{proof}
        \end{theorem}

\section{From vector bundles to coherent sheaves}\label{section:from-vector-bundles-to-coherent-sheaves}

    \subsection{The whole story: a summary}\label{subsection:the-whole-story}

        Starting with some complex of coherent sheaves $\mathscr{F}^\anotherbullet\in\gcohX$, with the category of complexes of coherent sheaves being defined as in \cite[§3.1]{Hosgood2021}, we take Green's resolution, which produces a complex $\mathcal{E}^{\bullet,\anotherbullet}$ of vector bundles on the nerve (of some possible refined cover) resolving $i^*\mathscr{F}^\anotherbullet$.
        Then, also by Green's resolution, we get simplicial connections on each of the $\mathcal{E}^{\bullet,i}$, and \cite[Theorem~4.29]{Hosgood2021} tells us that these are admissible.\footnote{Indeed, they are all of the form of barycentric connections, i.e. \emph{generated in degree zero} (cf. \cite[§4.5]{Hosgood2021}).}
        So, applying the generalised invariant polynomial $P_\bullet = \{\tr\circ\mu_n\}_{n\in\mathbb{N}}$ (where $\mu_n$ is the multiplication map that sends an $n$-fold tensor product of endomorphisms to the endomorphism given by composition (resp. wedge product)) to the curvatures of each of the simplicial connections, we obtain the simplicial exponential (resp. standard) Atiyah classes of each $\mathcal{E}^{\bullet,i}$.
        Using the alternating-sum convention, this gives us the simplicial exponential (or standard) Atiyah class of $\mathscr{F}$.
        Finally, by fibre integration of the trace of these classes, we obtain closed classes in the Čech-de~Rham bicomplex, and thus classes in de~Rham (or even tDR) cohomology.

        \begin{lemma}
            This construction is independent of the choice of twisting cochain (and thus cover) and of local connections.
        \end{lemma}
        \begin{proof}
            The argument for this is of exactly the same form as the proof of \cite[Theorem~2.4]{Green1980}, where we pass to the disjoint union of the covers and can inductively construct a twisting cochain which is a common extension of the two twisting cochains from which we start.
            Although Green only gives the proof for the $(p,p)$-term, his argument applies equally well to the full Čech-de~Rham representative $(C^{p,p},C^{p-1,p+1},\ldots,C^{0,2p})$ without modification.
        \end{proof}

        \begin{remark}
            Of particular importance is \cite[Lemma~4.39]{Hosgood2021}, which tells us that the homotopy colimit (over refinements of covers) of the localisation of the category $\sgreenzeroX$ of Green complexes is equivalent to the homotopy colimit of the localisation of the category of complexes $\greenzeroX$ which are locally quasi-isomorphic to Green complexes.
            In particular, we only know how to construct admissible simplicial connections for objects of the former, and so it is necessary that our construction of $\mathcal{E}^{\bullet,\anotherbullet}$ in the above is an object of $\sgreenzeroX$.
        \end{remark}

    \subsection{General properties}

        \begin{lemma}\label{lemma:agrees-for-line-bundles}
            Let $\mathcal{L}$ be a holomorphic line bundle with transition maps $\{g_{\alpha\beta}\}\in\check{\mathcal{C}}^1(\mathbb{C}^\times)$.
            Then
            \begin{equation*}
                \tr\int_{|\Delta^\bullet|}\simpexpat{1}_\mathcal{L} = \mathrm{c}_1(\mathcal{L}),
            \end{equation*}
            where $\mathrm{c}_1(\mathcal{L})$ denotes the first Chern class of the line bundle given by the connecting homomorphism from the Picard group in the long exact sequence associated to the \define{exponential sheaf sequence}
            \begin{equation*}
                0
                \to
                2\pi i\mathbb{Z}
                \hookrightarrow
                \OO_X
                \xtwoheadrightarrow{\exp}
                \OO_X^\times
                \to
                0
            \end{equation*}
            (that is, the classical definition, as given in e.g. \cite[Definition~2.2.13]{Huybrechts2005}).
            \begin{proof}
                We have already calculated in \cref{equation:omega-alpha-beta-explicitly} that the left-hand side is equal to $M_{\alpha\beta}^{-1}\d M_{\alpha\beta}$.
                That the right-hand side gives the same result can be found in e.g. the proof of the Proposition in \cite[Chapter~1, §1, \emph{Chern Classes of Line Bundles}, p.~141]{Griffiths&Harris1994}.
            \end{proof}
        \end{lemma}

        \begin{lemma}\label{lemma:functorial-under-pullbacks}
            Let $f\colon Y\to X$ be a morphism of complex-analytic manifolds, and let $\mathscr{F}$ be a coherent sheaf on $X$.
            Then
            \begin{equation*}
                f^*\left(
                    \int_{|\Delta^\bullet|} \tr \simpexpat{k}_{\mathscr{F}}
                \right)
                =
                \int_{|\Delta^\bullet|} \tr \simpexpat{k}_{f^*\mathscr{F}}
            \end{equation*}
            for all $k\in\mathbb{N}$, where $f^*$ denotes the \emph{derived} pullback.
            \begin{proof}
                This is basically the combination of the following facts: the derived pullback is exact; the simplicial Atiyah class is defined by taking a resolution; the derived pullback of a locally free sheaf agrees with the non-derived pullback; the non-simplicial Atiyah class of the non-derived pullback of a locally free sheaf is exactly the pullback of the non-simplicial Atiyah class of the locally free sheaf.
                We spell out how to join up these facts in slightly more detail below.

                Since the derived pullback is exact, we know that, given Green's resolution $\mathcal{E}^{\bullet,\anotherbullet}$ of $\mathscr{F}^\bullet = (\nerve{\bullet}\to X)^*\mathscr{F}$ by vector bundles on the nerve, the derived pullback $f^*\mathcal{E}^{\bullet,\anotherbullet}$ is a resolution for $f^*\mathscr{F}^\bullet$.
                But the derived pullback on locally free objects agrees with the non-derived pullback (since we are tensoring with something locally free), and so $f^*\mathcal{E}^{\bullet,\anotherbullet}$ can be calculated by the non-derived pullback.
                Now we can use the fact\footnote{We are working with vector bundles, which are locally free, and so, as previously mentioned, the derived pullback is just exactly the non-derived pullback. There is some subtlety however, since we are using the word ``pullback'' to mean a few different things here. When we talk about pulling back the Atiyah class of $E$, we mean first applying the pullback of sheaves, to get some class in $\HH^1(X,f^*\Omega_X^1\otimes\footnotesizesheafend(f^*E))$ (using the fact that pullbacks commute with $\footnotesizesheafend$ for finite-dimensional locally free sheaves), and then applying the canonical map $f^*\Omega_X^1\to\Omega_Y^1$ (which is what we really mean when we talk about pulling back forms). The fact that this construction sends $\at_E$ to $\at_{f^*E}$ follows from the fact that short exact sequences are distinguished triangles, and so specifying a morphism between the first and last (non-zero) terms of two SESs extends uniquely to a morphism of the SESs (in that we get a unique morphism between the middle terms of the two SESs such that everything commutes).} that, for a single vector bundle, the Atiyah class of the pullback is the pullback of the Atiyah class; for a single vector bundle on the nerve, the simplicial Atiyah class is determined entirely by the $\omega_{\alpha_0\alpha_i}$ (which represent the non-simplicial Atiyah class).
                Combining all the above, we see that all the forms defining the simplicial Atiyah class of the derived pullback of $\mathscr{F}$ are exactly the pullbacks of the forms defining the simplicial Atiyah class of $\mathscr{F}$.
                In particular, then, taking the trace and then fibre integrating (both of which commute with the pullback of forms on $X$) gives us the required result.
            \end{proof}
        \end{lemma}

        \begin{lemma}\label{lemma:additive-on-split-implies-additive-on-short}
            If the total simplicial Atiyah classes are additive on every split exact sequence of coherent sheaves then they are additive on every short exact sequence of coherent sheaves.
            \begin{proof}
                Let $0\to\mathscr{F}\xrightarrow{\iota}\mathscr{G}\xrightarrow{\pi}\mathscr{H}\to0$ be a short exact sequence of coherent sheaves on $X$, and $t\in\mathbb{C}$ (which can be thought of as $t\in\Gamma(\mathbb{C},\mathbb{P}^1)$).
                Write $p\colon X\times\mathbb{P}^1\to X$ to mean the projection map.
                Define
                \begin{equation*}
                    \mathcal{N} = \Ker\big(p^*\mathscr{G}(1)\oplus p^*\mathscr{H}\xrightarrow{\pi(1)-t\cdot\id}p^*\mathscr{H}(1)\big)
                \end{equation*}
                where $(\pi(1)-t\cdot\id)\colon (g\otimes y, h)\mapsto \pi(g)\otimes y - h\otimes t$.
                We claim that this gives a short exact sequence
                \begin{equation*}
                    0\to p^*\mathscr{F}(1)\to\mathcal{N}\to p^*\mathscr{H}\to0
                \end{equation*}
                of sheaves over $X\times\mathbb{P}^1$, where the maps are the ``obvious'' ones: $p^*\mathscr{F}(1)\to\mathcal{N}$ is the map $\iota(1)\colon p^*\mathscr{F}(1)\to p^*\mathscr{G}(1)$ included into $p^*\mathscr{G}(1)\oplus p^*\mathscr{H}$ (which we prove lands in $\mathcal{N}$ below); and $\mathcal{N}\to p^*\mathscr{H}$ is the projection $p^*\mathscr{G}(1)\oplus p^*\mathscr{H}\to p^*\mathscr{H}$ restricted to $\mathcal{N}$.

                To prove surjectivity, let $h\in\Gamma(U,p^*\mathscr{H})$.
                Then $h\otimes t\in\Gamma(U,p^*\mathscr{H}(1))$.
                But \mbox{$\pi\colon\mathscr{G}\to\mathscr{H}$} is surjective, and thus so too is the induced map $\pi(1)\colon p^*\mathscr{G}(1)\to p^*\mathscr{H}(1)$.
                Hence there exists $g\otimes y\in\Gamma(U,p^*\mathscr{G}(1))$ such that $\pi(g\otimes y)=h\otimes t$.
                So $(g\otimes y, h)\in\mathcal{N}$ maps to $h$.

                To prove injectivity (and that this map is indeed well defined), we use the fact that tensoring with $\mathscr{O}(1)$ is exact, and so, in particular, $\iota(1)\colon p^*\mathscr{F}(1)\to p^*\mathscr{G}(1)$ is injective.
                The inclusion into the direct sum $p^*\mathscr{G}(1)\oplus p^*\mathscr{H}$ is injective by the definition of a direct sum, so all that remains to show is that the image of this composite map is contained inside $\mathcal{N}$.
                Let $f\otimes x\in\Gamma(U,p^*\mathscr{F}(1))$.
                Then this maps to $(\iota(f)\otimes x,0)\in p^*\mathscr{G}(1)\oplus p^*\mathscr{H}$, but this is clearly in the kernel of $\pi(1)-t\cdot\id$ since $\pi\iota(f)=0$.

                To prove exactness, it suffices to show that $\Ker(\mathcal{N}\to p^*\mathscr{F})\cong p^*\mathscr{F}(1)$.
                But
                \begin{align*}
                    \Ker(\mathcal{N}\to p^*\mathscr{F}) &= \{(g\otimes y, h)\in\mathcal{N} \mid h=0\}\\
                    &= \{(g\otimes y, h)\in p^*\mathscr{G}(1)\oplus p^*\mathscr{H} \mid h=0\text{ and }\pi(g)\otimes y-h\otimes t=0\}\\
                    &= \{(g\otimes y, h)\in p^*\mathscr{G}(1)\oplus p^*\mathscr{H} \mid \pi(g)\otimes y=0\}\\
                    &= \{(g\otimes y, h)\in p^*\mathscr{G}(1)\oplus p^*\mathscr{H} \mid (g\otimes y)\in\Im\iota(1)\}\\
                    &\cong p^*\mathscr{F}(1).
                \end{align*}

                Now we claim that the short exact sequence is split for $t=0$, and has $\mathcal{N}\cong p^*\mathscr{G}$ for $t\neq0$.
                Formally, we do this by looking at the pullback of the map $X\times\{t\}\to X\times\mathbb{P}^1$, but we can think of this as just ``picking a value for $t$''.

                For $t=0$, by definition,
                \begin{align*}
                    \mathcal{N} &= \Ker\big(p^*\mathscr{G}(1)\oplus p^*\mathscr{H}\xrightarrow{\pi(1)-t\cdot\id}p^*\mathscr{H}(1)\big)\\
                    &= \Ker\big(p^*\mathscr{G}(1)\oplus p^*\mathscr{H}\xrightarrow{\pi(1)}p^*\mathscr{H}(1)\big)\\
                    &\cong \Ker\big(p^*\mathscr{G}(1)\xrightarrow{\pi(1)}p^*\mathscr{H}(1)\big) \oplus p^*\mathscr{H}\\
                    &\cong p^*\mathscr{F}(1)\oplus p^*\mathscr{H}.
                \end{align*}
                For $t\neq0$, define the injective morphism $\varphi\colon p^*\mathscr{G}\to\mathcal{N}$ of coherent sheaves by $\varphi(g)=(g\otimes t,\pi(g))$.
                To see that this is also surjective, let $(g\otimes y, h)\in\mathcal{N}$.
                If $y=0$ then we must have $h=0$, and so $(g\otimes y, h) = (0,0) = \varphi(0\otimes0)$.
                If $y\neq0$ then $\pi(g)\otimes y-h\otimes t=0$, with $y,t\neq0$, whence $\pi(g)=\frac{y}{t}h$.
                Then $(g\otimes y, h) = (\frac{t}{y}g\otimes t, \pi(\frac{t}{y}g)) = \varphi(\frac{t}{y}g).$

                As one final ingredient, note that any coherent sheaf on $X$ pulled back to a sheaf on $X\times\mathbb{P}^1$ is flat over $\mathbb{P}^1$, and so $\mathcal{N}$ is flat over $\mathbb{P}^1$, since both $\mathscr{F}(1)$ and $\mathscr{H}$ are.
                Thus, for $\tau_t\colon X\times\{t\}\to X\times\mathbb{P}^1$ given by a choice of $t\in\mathbb{C}$, the derived pullback $\mathbb{L}\tau_t^*\mathcal{N}$ agrees with the usual pullback $\tau_t^*\mathcal{N}$.

                Now we use the \emph{$\mathbb{P}^1$-homotopy invariance} of de~Rham cohomology, which is the following statement: the induced map
                \begin{equation*}
                    \tau_t^*\colon\HH^\bullet\big(X\times\mathbb{P}^1,\Omega_{X\times\mathbb{P}^1}^\bullet\big) \to \HH^\bullet\big(X\times\{t\},\Omega_{X\times\{t\}}^\bullet\big)
                \end{equation*}
                does \emph{not} depend on the choice of $t$.
                Since $X\times\{t\}$ is (canonically) homotopic to $X$, we can identify $(p\tau_t)^*$ with the identity on $\HH^\bullet(X,\Omega_X^\bullet)$.
                Since the SES splits for $t=0$, by our hypothesis, flatness, and the fact that Green's construction is functorial under derived pullback, we know that
                \begin{align*}
                    \tau_0^*c(\mathcal{N}) = c(\mathbb{L}\tau_0^*\mathcal{N}) = c(\tau_0^*\mathcal{N}) &= c(\tau_0^*p^*\mathscr{F}(1)\oplus\tau_0^*p^*\mathscr{H})\\
                    &= c(\mathscr{F})\wedge c(\mathscr{H}).
                \end{align*}
                But we also know that $\mathcal{N}\cong p^*\mathscr{G}$ for $t\neq0$, and so
                \begin{equation*}
                    c(\mathscr{G}) = (p\tau_t)^*c(\mathscr{G}) = \tau_t^*c(\mathcal{N}).
                \end{equation*}
                So, finally, the $t$-invariance of $\tau^*$ tells us that
                \begin{equation*}
                    c(\mathscr{G}) = c(\mathscr{F})\wedge c(\mathscr{H}).\qedhere
                \end{equation*}
            \end{proof}
        \end{lemma}

        \begin{lemma}\label{lemma:additive-on-split}
            The total simplicial Atiyah classes are additive on every split exact sequence of coherent sheaves.
            \begin{proof}
                This is exactly \cite[Lemma~2.6]{Green1980}, but the essence of the proof is simple: twisting cochains behave nicely with respect to direct sums, as do all of the constructions giving the simplicial connections generated in degree zero.
            \end{proof}
        \end{lemma}

        \begin{corollary}
            The total simplicial Atiyah classes are additive on every short exact sequence of coherent sheaves.
        \end{corollary}

        \begin{proof}
            This is simply the syllogism of \cref{lemma:additive-on-split-implies-additive-on-short} and \cref{lemma:additive-on-split}.
        \end{proof}

    \subsection{The compact case}\label{subsection:the-compact-case}

        In the case where $X$ is compact, we can appeal to \cite[Theorem~6.5]{Grivaux2009}, which gives a list of conditions that, if satisfied, ensure uniqueness of Chern classes: if we can show that the fibre integrals of traces of the simplicial Atiyah classes satisfy certain axioms, then we know that they are exactly the same classes as given by any other construction of Chern classes.

        Since tDR cohomology satisfies conditions ($\alpha$) to ($\delta$) of \cite[§6.2]{Grivaux2009}, it suffices to check that the following things hold true:
        \begin{enumerate}[(i)]
            \item \emph{The construction agrees with `the' classical one for line bundles.}
                Although \cite[Lemma~2.5]{Green1980} tells us that the construction actually agrees with the classical one for \emph{arbitrary} vector bundles, we explain the specific case of line bundles in \cref{lemma:agrees-for-line-bundles}.
            \item \emph{The construction is functorial under pullbacks.}
                This is \cref{lemma:functorial-under-pullbacks}.
            \item \emph{The construction gives us the Whitney sum formula for short exact sequences.}
                We have split the proof of this into two steps: firstly, showing in \cref{lemma:additive-on-split-implies-additive-on-short} that it suffices to prove that we have the Whitney sum formula for \emph{split} exact sequences; then showing in \cref{lemma:additive-on-split} that we do indeed have the Whitney sum formula for split exact sequences.
            \item \emph{The construction gives us the Grothendieck-Riemann-Roch formula for closed immersions.}
                Although this is given as a requirement in \cite[Theorem~6.5]{Grivaux2009}, it is actually not necessary, thanks to e.g. \cite[Proposition~3.1]{Grivaux2012}, whose proof relies only on the three properties above, following some classical algebraic geometry methods (such as deformation to the normal cone).
        \end{enumerate}

        \begin{remark}
            Again, when $X$ is compact, \cite[Lemma~2.7]{Green1980} gives a direct proof of the fact that the $(p,p)$-term of the fibre integral of the trace of these simplicial Atiyah classes agrees (up to a constant factor) with the Chern classes of Atiyah-Hirzenbruch.
        \end{remark}

        \begin{corollary}
        \label{corollary:main-corollary}
            If $X$ is a compact complex-analytic manifold, then, for any coherent analytic sheaf, the Chern classes given by the fibre integrals of the traces of the simplicial Atiyah classes (following the construction described in \cref{subsection:the-whole-story}) agree with all other constructions of Chern classes of coherent analytic sheaves.
        \end{corollary}

\addcontentsline{toc}{section}{References}
\printbibliography

@book{Ble1981,
  title = {Gauge Theory and Variational Principles},
  author = {Bleecker, David},
  date = {1981},
  series = {Global {{Analysis Pure}} and {{Applied}}: Series {{A}}},
  volume = {1},
  publisher = {{Addison–Wesley Publishing Co.}},
  isbn = {0-201-10096-7},
}

@article{Dupont1976,
  title = {Simplicial de Rham cohomology and characteristic classes of flat bundles},
  author = {Dupont, Johan L},
  date = {1976},
  journal = {Topology},
  volume = {15},
  pages = {233–245},
  doi = {10.1016/0040-9383(76)90038-0}
}

@thesis{Green1980,
  title = {Chern classes for coherent sheaves},
  author = {Green, H I},
  date = {1980},
  institution = {University of Warwick},
  url = {https://pugwash.lib.warwick.ac.uk/record=b1751746~S1},
  type = {PhD Thesis}
}

@book{Griffiths&Harris1994,
  title = {Principles of {{Algebraic Geometry}}},
  author = {Griffiths, Phillip and Harris, Joseph},
  date = {1994},
  edition = {Wiley Classics Library},
  publisher = {{John Wiley \& Sons, Inc.}},
  isbn = {978-0-471-05059-9}
}

@thesis{Grivaux2009,
  title = {Some problems in complex and almost-complex geometry},
  author = {Grivaux, Julien},
  date = {2009},
  institution = {Université Pierre-et-Marie-Curie (Paris 6)},
  eprint = {tel-00460334},
  eprinttype = {HAL}
}

@article{Grivaux2012,
  title = {On a conjecture of Kashiwara relating Chern and Euler Classes of O-modules},
  author = {Grivaux, Julien},
  date = {2012},
  journal = {Journal of Differential Geometry},
  volume = {90},
  pages = {267--275},
  doi = {10.4310/jdg/1335230847},
}

@thesis{Hosgood2020,
  title = {Chern Classes of Coherent Analytic Sheaves: A Simplicial Approach},
  author = {Hosgood, Timothy},
  date = {2020},
  institution = {{Université d'Aix-Marseille}},
  eprint = {tel-02882140},
  eprinttype = {HAL}
}

@article{Hosgood2021,
  title = {Simplicial {{Chern}}-{{Weil}} theory for coherent analytic sheaves, part {{I}}},
  author = {Hosgood, Timothy},
  date = {2023},
  journal = {{Bulletin de la Société Mathématique de France}},
  volume = {151},
  pages = {127--170},
  doi = {10.24033/bsmf.2866},
}

@book{Huybrechts2005,
  title = {Complex Geometry: An Introduction},
  shorttitle = {Complex Geometry},
  author = {Huybrechts, Daniel},
  date = {2005},
  publisher = {Springer-Verlag},
  doi = {10.1007/b137952},
  isbn = {978-3-540-21290-4},
  series = {Universitext}
}

@article{OTT1985,
  title = {A Grothendieck-Riemann-Roch formula for maps of complex manifolds},
  author = {O'Brian, Nigel R and Toledo, Domingo and Tong, Yue Lin L},
  date = {1985},
  journaltitle = {Mathematische Annalen},
  volume = {271},
  number = {4},
  pages = {493-526},
  doi = {10.1007/BF01456132}
}

@incollection{TT1986,
  title = {Green’s theory of Chern classes and the Riemann--Roch formula},
  booktitle = {The Lefschetz Centennial Conference. Part I: Proceedings on Algebraic Geometry},
  author = {Toledo, Domingo and Tong, Yue Lin L},
  date = {1986},
  series = {Contemporary Mathematics},
  volume = {58.I},
  pages = {261--275},
  publisher = {American Mathematical Society},
  doi = {10.1090/conm/058.1/860421},
}

@book{Voisin2002a,
  title = {Hodge Theory and Complex Algebraic Geometry I},
  author = {Voisin, Claire},
  date = {2008},
  volume = {1},
  publisher = {Cambridge University Press},
  isbn = {978-0-521-71801-1},
  number = {76},
  series = {Cambridge Studies in Advanced Mathematics},
  translator = {Schneps, Leila}
}

\end{document}